\documentclass{amsart}
\usepackage{latexsym,amsxtra,amscd,ifthen}
\usepackage{amsfonts}
\usepackage{verbatim}
\usepackage{amsmath}
\usepackage{amsthm}
\usepackage{amssymb}

\numberwithin{equation}{section}

\theoremstyle{plain}
\newtheorem{theorem}{Theorem}[section]
\newtheorem{lemma}[theorem]{Lemma}
\newtheorem{proposition}[theorem]{Proposition}

\theoremstyle{definition}
\newtheorem{definition}[theorem]{Definition}

\newtheorem{remark}[theorem]{Remark}

\makeatletter              
\let\c@equation\c@theorem 

\makeatother

\newcommand{\AAA}{\mathbb{A}}
\newcommand{\BB}{\mathbb{B}}
\newcommand{\CC}{\mathbb{C}}
\newcommand{\DD}{\mathbb{D}}
\newcommand{\EE}{\mathbb{E}}
\newcommand{\FF}{\mathbb{F}}
\newcommand{\GG}{\mathbb{G}}
\newcommand{\HH}{\mathbb{H}}
\newcommand{\II}{\mathbb{I}}
\newcommand{\JJ}{\mathbb{J}}
\newcommand{\KK}{\mathbb{K}}
\newcommand{\LL}{\mathbb{L}}
\newcommand{\MM}{\mathbb{M}}
\newcommand{\NN}{\mathbb{N}}
\newcommand{\OO}{\mathbb{O}}
\newcommand{\PP}{\mathbb{P}}
\newcommand{\QQ}{\mathbb{Q}}
\newcommand{\RR}{\mathbb{R}}
\newcommand{\SSS}{\mathbb{S}}
\newcommand{\TT}{\mathbb{T}}
\newcommand{\UU}{\mathbb{U}}
\newcommand{\VV}{\mathbb{V}}
\newcommand{\WW}{\mathbb{W}}
\newcommand{\XX}{\mathbb{X}}
\newcommand{\YY}{\mathbb{Y}}
\newcommand{\ZZ}{\mathbb{Z}}
\newcommand{\LIST}{\mathcal{LIST}}

\DeclareMathOperator{\Ext}{Ext}

\DeclareMathOperator{\gr}{gr}

\DeclareMathOperator{\Aut}{Aut}

\DeclareMathOperator{\GKdim}{GKdim}
\DeclareMathOperator{\Spec}{Spec}
\DeclareMathOperator{\End}{End}

\DeclareMathOperator{\Hom}{Hom}

\begin{document}

\title
{Double Extension Regular Algebras of Type (14641)}

\author{James J. Zhang and Jun Zhang}

\address{Department of Mathematics, Box 354350,
University of Washington, Seattle, WA 98195, USA}

\email{zhang@math.washington.edu}
\email{junz@math.washington.edu}

\begin{abstract}
We construct several families of Artin-Schelter regular 
algebras of global dimension four using double Ore extension 
and then prove that all these algebras are strongly noetherian, 
Auslander regular, Koszul and Cohen-Macaulay domains.
Many regular algebras constructed in the paper are new and 
are not isomorphic to either a normal extension or an Ore 
extension of an Artin-Schelter regular algebra of global 
dimension three. 
\end{abstract}



\maketitle


\setcounter{section}{-1}
\section{Introduction}
\label{xxsec0}

One of the most important projects in noncommutative algebraic 
geometry is the classification of noncommutative projective 
3-spaces, or quantum ${\mathbb P}^3$s. An algebraic version of
this project is the classification of Artin-Schelter regular 
algebras of global dimension four. There has been extensive 
research on Artin-Schelter regular algebras of global dimension 
four; and many families of regular algebras have 
been discovered in recent years \cite{LSV, LPWZ, Sk1, Sk2, SS, 
VV1, VV2, VVW, Va1, Va2}. The main goal of this paper is 
to construct and study a large class of new Artin-Schelter 
regular algebras of dimension four, called double Ore extensions.

The notion of double Ore extension (or double extension for the
rest of the paper) was introduced in \cite{ZZ}. A double 
extension of an algebra $A$ is denoted by 
$A_P[y_1,y_2;\sigma, \delta,\tau]$ and the meanings of DE-data 
$\{P,\sigma,\delta, \tau\}$ will be reviewed in Section 
\ref{xxsec1}. A more general way of building regular algebra 
of dimension four was presented by Caines in his Thesis \cite{Ca}. 
In principle, all double extensions in this paper are also 
``skew-polynomial rings'' in the sense of Caines. The idea
of double extensions were used by Patrick \cite{Pa} and 
Nyman \cite{Ny} in a different context. As a generalization of 
the classical Ore extension \cite{Or} the method of double Ore 
extension is simple and effective. By using the double extension
we construct regular algebras of global dimension four explicitly.

Many researchers have noted that the data associated to 
the whole class of regular algebras of dimension four are 
tremendous; and one needs to introduce some invariants to 
distinguish these algebras. For simplicity we only consider 
regular algebras of dimension four that are generated 
in degree 1. By the work of \cite{LPWZ}, such an algebra $B$ is 
generated by either 2, or 3, or 4 elements and the projective 
resolution of the trivial module $k_B$ is given in 
\cite[Proposition 1.4]{LPWZ}. When $B$ is generated by 4 elements, 
then the projective resolution of the trivial module $k_B$ is
of the form
$$0\to B(-4)\rightarrow B(-3)^{\oplus 4}\rightarrow B(-2)^{\oplus 6}
\rightarrow B(-1)^{\oplus 4}\rightarrow B\rightarrow k_B\to 0.
$$
Suggested by the form of the above resolution, we say such an 
algebra is {\it of type (14641)}. In this paper we mainly 
deal with algebras of type (14641). An algebra of type (14641)
is Koszul. 

With some help of Mathematical software Maple, we are able to 
classify all double extensions $A_{P}[y_1,y_2;\sigma]$ (with 
$\delta=0$ and $\tau=(0,0,0)$) of type (14641). Since Ore extensions 
and normal extensions of regular algebras of dimension three are 
well-understood and well-studied \cite{LSV}, we omit some of 
those from our classification. Our then ``partial'' classification 
consists of 26 families of regular algebras of type (14641); and 
it provides enough information to prove part (a) of the following 
theorem. 

\begin{theorem}
\label{xxthm0.1}
Let $B$ be a connected graded algebra generated by 
four elements of degree 1. Suppose that $B$ is a double extension 
$A_P[y_1,y_2;\sigma,\delta,\tau]$ where $A$ is an Artin-Schelter 
regular algebra of dimension 2.
\begin{enumerate}
\item
$B$ is a strongly noetherian, Auslander regular and 
Cohen-Macaulay domain.
\item
$B$ is of type (14641). As a consequence, $B$ is Koszul. 
\item
If $B$ is not isomorphic to an Ore extension of an 
Artin-Schelter regular algebra
of dimension three, then the trimmed double
extension $A_P[y_1,y_2;\sigma]$ (by setting $\delta=0$ and
$\tau=(0,0,0)$) is isomorphic to one of 26 
families listed in Section 4.
\end{enumerate}
\end{theorem}

These 26 families of algebras in part (b) are labeled by 
${\mathbb A}, {\mathbb B}, \cdots, {\mathbb Z}$. Let
$\LIST$ denote the class consisting of all algebras in
the families from $\AAA$ to $\ZZ$. 
Regular algebras of dimension three are well-understood 
\cite{AS,ATV1,ATV2}. Hence, in theory, Ore extensions of regular 
algebra of dimension three are well understood. That is our 
rational to omit Ore extensions in part (b). Besides, there 
are too many double extensions 
$A_{P}[y_1,y_2,\sigma,\delta,\tau]$ to list if we want to 
include all Ore extensions. On the other hand, in these 26 
families, many of the double extensions $A_P[y_1,y_2;\sigma]$ 
are still Ore extensions. The reason that we do not remove 
those Ore extensions is that there might be nonzero $\delta$ 
and $\tau$ such that $A_P[y_1,y_2;\sigma,\delta,\tau]$ 
(with the same $(P,\sigma)$) is not an Ore extension. In other 
words our classification is basically the classification of 
$(P,\sigma)$ so that $A_{P}[y_1,y_2;\sigma,\delta,\tau]$ is 
not an Ore extension for possible $(\delta,\tau)$. 

We want to remark that the software Maple is used in an 
elementary way only to reduce the length of the computation 
and all computation can be done by hand without assistant of 
Maple. Further, the regularity and other properties of every 
algebra is verified rigorously by other means. 

As a consequence of Theorem \ref{xxthm0.1}, for any algebra 
$A$ in the $\LIST$, the scheme of point modules (respectively, 
line modules) over $A$ is a genuine commutative projective
scheme \cite[Corollary E4.11]{AZ}. It would be very interesting to work out
geometric properties and geometry invariants (such as 
the point-scheme and the line-scheme) associated to $A$. 
There are also various algebraic questions we do not pursue
in this paper. For example, for any algebra $A$ in the 
$\LIST$, one may ask:
\begin{enumerate}
\item[(a)]
Is $A$ primitive? Does $A$ satisfies a polynomial identity?
What is the prime spectrum $\Spec A$?
\item[(b)]
What is the group of graded algebra automorphisms of $A$
(denoted by $\Aut(A)$)? Is there a non-trivial 
finite subgroup $G\subset \Aut(A)$ such that
$A^G$ is Artin-Schelter regular?
\item[(c)]
What invariants can be defined for the quotient division 
algebra of $A$? Is the quotient division algebra of $A$ 
always generated by two elements?
\end{enumerate}
Some of these questions are easy for each individual algebra; 
however, it could be a challenge to find a general approach that 
works for all algebras. Question (b) leads to finding more 
regular algebras of dimension four that may not generated in 
degree 1. 

Double extensions appear naturally in some slightly different 
contexts about regular algebras of type (14641). 

\begin{theorem}
\label{xxthm0.2}
Let $B$ be a noetherian Artin-Schelter regular algebra of type 
(14641). Suppose that $B$ is ${\mathbb Z}^2$-graded with a 
decomposition $B_1=B_{01}\oplus B_{10}$ where $B_{01}$ and 
$B_{10}$ are nonzero ${\mathbb Z}^2$-homogeneous components.
\begin{enumerate}
\item
If $\dim B_{01}=1$ or $\dim B_{10}=1$, then $B$ is isomorphic
to an Ore extension $A[y;\sigma]$ for some Artin-Schelter 
regular algebra $A$ of dimension three.
\item
If $\dim B_{01}=\dim B_{10}=2$, then $B$ is isomorphic to a 
trimmed double extension $A_{P}[y_1,y_2;\sigma]$ for some 
Artin-Schelter regular algebra $A$ of dimension two.
\end{enumerate}
\end{theorem}

Since all double extensions $A_{P}[y_1,y_2;\sigma]$ (that are 
not Ore extension of regular algebras of dimensional three) are 
classified in Section 4, Theorem \ref{xxthm0.2} gives a 
classification of (non-trivially) ${\mathbb Z}^2$-graded noetherian 
regular algebras of type (14641). Various other properties related 
to Artin-Schelter regular algebras are studied. Here is another 
characterization of the double extensions in Theorem \ref{xxthm0.1}. 

\begin{proposition}
\label{xxprop0.3}  
Let $B$ be an Artin-Schelter regular domain of global dimension 
four generated by four degree 1 elements. Then $B$ is a double 
extension if and only if there are $x_1,x_2\in B_1$ such that 
\begin{enumerate}
\item
$B$ has a quadratic relation involving only $x_1,x_2$,
\item
$B/(x_1,x_2)$ is Artin-Schelter regular of dimension two.
\end{enumerate}
\end{proposition}

Here is an outline of the paper. In Section 1 we review some
basic definitions. Theorem \ref{xxthm0.2} and Proposition 
\ref{xxprop0.3} are proved in Section 2. Sections 3 and 4 are
devoted to the classification that is unfortunately very tedious.
Our main theorem \ref{xxthm0.1} is proved in Section 5.

\section{Definitions}
\label{xxsec1}

Throughout $k$ is a commutative base field, that is algebraically 
closed. Everything is over $k$; in particular, an algebra or a ring 
is a $k$-algebra. An algebra $A$ is called {\it connected graded} if 
$$A=k\oplus A_1\oplus A_2\oplus \cdots$$
with $1\in k=A_0$ and $A_iA_j\subset A_{i+j}$ for all $i,j$. If $A$ 
is connected graded, then $k$ also denotes the trivial graded module 
$A/A_{\geq 1}$. In this paper we are working on connected graded 
algebras. One basic concept we will use is the Artin-Schelter 
regularity, which we now review. A connected graded algebra $A$ is 
called {\it Artin-Schelter regular} or {\it regular} for short if 
the following three conditions hold.
\begin{enumerate}
\item[(AS1)]
$A$ has finite global dimension $d$, and
\item[(AS2)] 
$A$ is {\it Gorenstein}, namely, there is an
integer $l$ such that,
$$\Ext^i_A({_Ak}, A)=\begin{cases} k(l) & \text{ if }
i=d\\
                                0   & \text{ if }
i\neq d
\end{cases}
$$
where $k$ is the trivial $A$-module; and the same condition holds 
for the right trivial $A$-module $k_A$.
\item[(AS3)]
$A$ has finite Gelfand-Kirillov dimension, i.e., there
is a positive number $c$ such that $\dim A_n< c\; n^c$ for
all $n\in \mathbb{N}$.
\end{enumerate}
If $A$ is regular, then the global dimension of $A$ is called the
{\it dimension} of $A$. The notation $(l)$ in (AS2) is the $l$-th 
degree shift of graded modules.

\begin{definition}
\label{xxdefn1.1}
Let $A, B$ and $C$ be connected graded algebras. 
The algebra $B$ is called an {\it extension} of $(A|C)$, if there 
is a sequence of graded maps
$$0\to A\xrightarrow{f} B\xrightarrow{g}C\to 0$$
satisfying the following conditions:
\begin{enumerate}
\item
Both $f$ and $g$ are graded algebra homomorphisms.
\item
$B$ contains $A$ as a graded subalgebra via $f$.
\item
$A_{\geq 1}B=BA_{\geq 1}$ and the map $g$ induces an isomorphism
of graded algebras $B/(A_{\geq 1})\cong C$.
\item
There is a vector space $\bar{C}\subset B$ such that $g:\bar{C}
\to C$ is an isomorphism of graded vector space and $B$ is a left 
and a right free $A$-module with basis $\bar{C}$.
\end{enumerate}
If $C$ is a regular algebra of dimension $n$, we call $B$
an $n$-extension of $A$. If $A$ is a regular algebra of 
dimension $n$, then we call $B$ an $n$-co-extension of $C$. 
\end{definition}

The following lemma characterize $n$-(co)-extensions for 
$n=0,1$.

\begin{lemma}
\label{xxlem1.2} 
\begin{enumerate}
\item
$B$ is a $0$-extension of $A$ if and only if $f:A\to B$
is an isomorphism.
\item
$B$ is a $0$-co-extension of $C$ if and only if $g:B\to C$
is an isomorphism.
\item
Suppose $A$ is generated in degree 1. Then $B$ is a 
$1$-extension of $A$ if and only if $B$ is an Ore
extension of $A$.
\item
$B$ is a $1$-co-extension of $C$ if and only if $B$ is 
a normal extension of $C$, namely, there is a normal
regular element $t$ such that $B/(t)\cong C$.
\end{enumerate}
\end{lemma}

\begin{proof} (a,b) Trivial.

(c) If $B$ is an Ore extension $A[x;\sigma,\delta]$, then
it is easy to check that $B$ is a $1$-extension of $A$
(without assuming that $A$ is generated in degree 1).

Now we assume $B$ is a $1$-extension of $A$ and $A$ 
is generated by $A_1$. In this case $C=k[x]$. Let 
$\bar{C}=\bigoplus_{i\geq 0}k x_i$ such that $g(x_i)=x^i$. 
For every $a\in A_1$, $ax_1\in A=A\oplus 
\bigoplus_{i\geq 1} x_i A$ and we can write
$$ax_1=b_0+\sum_{i\geq 1}x_i b_i.$$
Since $\deg (ax_1)=1+\deg x_1\leq \deg x_2<\deg x_i$ for all
$i>2$, $b_i=0$ for all $i>2$ and $b_2\in k$. Since $A_{\geq 1}B
=BA_{\geq 1}$, $b_2=0$. Thus $ax_1\in A\oplus x_1 A$. This
implies that $Ax_1\subset A\oplus x_1A$. By symmetry,
$x_1A\subset A\oplus Ax_1$. So there are maps $\sigma$ and 
$\delta$ such that
$$x_1 r=\sigma(r) x_1+\delta(r)$$
for all $r\in A$. It is easy to see that 
$\sigma$ is an algebra automorphism of $A$ and $\delta$ is a 
$\sigma$-derivation of $A$. Finally it is routine to check 
that $B=A[x_1;\sigma,\delta]$.

(d) In this case $A=k[t]$. Since $B$ is a free $A$-module
on both sides, $t$ is regular on both sides. By (b), $A_{\geq 1}B=
BA_{\geq 1}$. This implies that $tB=Bt=(A_{\geq 1})$. So $t$ is a 
regular normal element of $B$ and $B/(t)=C$. This says that
$B$ is a normal extension of $C$. The converse is clear.
\end{proof}

If $A$ is not generated in degree 1, then it is possible 
that a $1$-extension is not an Ore extension. Most 
algebras in this paper will be generated in degree 1.
Some basic properties of Ore extensions can be found in
\cite[Chapter 1]{MR}.

Next we will show that the definition of a $2$-extension is
equivalent to that of a double extension introduced in 
\cite{ZZ}. We first review the definition of a 
double extension in the connected graded case.

\begin{definition}\cite[Definition 1.3]{ZZ}
\label{xxdefn1.3}
Let $A$ be a connected graded algebra and $B$ be another
connected graded algebra containing $A$ as a graded subring. 
\begin{enumerate}
\item
We say $B$ is a {\it right double extension} of $A$ if the 
following conditions holds.
\begin{enumerate}
\item[(ai)]
$B$ is generated by $A$ and two variables 
$y_1$ and $y_2$ of positive degree.
\item[(aii)]
$\{y_1,y_2\}$ satisfies a homogeneous relation
\begin{equation}
\label{R1}
y_2y_1=p_{12}y_1y_2+p_{11}y_1^2+\tau_{1}y_1+
\tau_{2}y_2+\tau_{0}
\tag{R1}
\end{equation}
where $p_{12},p_{11}\in k$ and $\tau_{1}, \tau_{2}, 
\tau_{0} \in A$. 
\item[(aiii)]
As a left $A$-module, $B=\sum_{n_1,n_2\geq 0} 
Ay_1^{n_1}y_2^{n_2}$ and it is a left free 
$A$-module with a basis $\{y_1^{n_1}y_2^{n_2}
\;|\; n_1\geq 0,n_2\geq 0\}$.
\item[(aiv)]
$y_1 A+y_2 A\subseteq A y_1 + A y_2 + A$
\end{enumerate}
Let $P$ denote the set of scalar parameters 
$\{p_{12},p_{11}\}$ and  let $\tau$ denote 
the set $\{\tau_{1},\tau_{2},\tau_{0}\}$. 
\item
We say $B$ is a {\it left double extension} of 
$A$ if the following conditions holds.
\begin{enumerate}
\item[(bi)]
$B$ is generated by $A$ and two variables 
$y_1$ and $y_2$.
\item[(bii)]
$\{y_1,y_2\}$ satisfies a homogeneous relation
\begin{equation}
\label{L1}
y_1y_2=p_{12}'y_2y_1+p_{11}'y_1^2+y_1\tau_{1}'+
y_2\tau_{2}'+\tau_{0}'
\tag{L1}
\end{equation}
where $p_{12}',p_{11}'\in k$ and $\tau_{1}',\tau_{2}', 
\tau_{0}' \in A$. 
\item[(biii)]
As a right $A$-module, $B=\sum_{n_1,n_2\geq 0} 
y_2^{n_1}y_1^{n_2}A$ and it is a right free 
$A$-module with a basis $\{y_2^{n_1}y_1^{n_2}
\;|\; n_1\geq 0,n_2\geq 0\}$.
\item[(biv)]
$Ay_1 +Ay_2 A\subseteq y_1 A +  y_2 A + A$.
\end{enumerate}
\item
We say $B$ is a {\it double extension} if it is a left and
a right double extension of $A$ with the same 
generating set $\{y_1,y_2\}$.
\end{enumerate}
\end{definition}

If $B$ is a double extension of $A$, then $p_{12}p'_{12}=1$
and hence $p_{12}\neq 0$. Both Definitions \ref{xxdefn1.1}
and \ref{xxdefn1.3} are abstract. To study 
extensions we need to find more precise information 
about these algebras. The condition in Definition 
\ref{xxdefn1.3}(aiv) can be written as follows:
\begin{equation}
\label{R2}
\begin{pmatrix} y_1\\y_2\end{pmatrix} r
:= \begin{pmatrix} y_1 r\\y_2 r\end{pmatrix} =
\begin{pmatrix} \sigma_{11}(r)& \sigma_{12}(r)\\
\sigma_{21}(r)&\sigma_{22}(r)\end{pmatrix} 
\begin{pmatrix} y_1\\y_2\end{pmatrix}
+\begin{pmatrix} \delta_1(r)\\ \delta_2(r)
\end{pmatrix}
\tag{R2}
\end{equation}
for all $r\in A$. Here $\sigma(r):=
\begin{pmatrix} \sigma_{11}(r)&
\sigma_{12}(r)\\
\sigma_{21}(r)&\sigma_{22}(r)\end{pmatrix}$ is an
algebra homomorphism from $A$ to $M_2(A)$ and
$\delta(r):=
\begin{pmatrix} \delta_1(r)\\ \delta_2(r)
\end{pmatrix}$ 
is a $\sigma$-derivation from $A$ to $A^{\oplus 2}:=
\begin{pmatrix} A\\A\end{pmatrix}$. By \cite[Section 1]{ZZ}, 
$\sigma$ and $\delta$ are uniquely determined. Together with 
Definition \ref{xxdefn1.3}, all symbols in the DE-data 
$\{P,\sigma,\delta,\tau\}$ are defined, and the double 
extension $B$ in Definition \ref{xxdefn1.3} is denoted by 
$A_P[y_1,y_2;\sigma,\delta,\tau]$. We call 
$\sigma$ a {\it homomorphism}, $\delta$ 
a {\it derivation}, $P$  a {\it parameter} and 
$\tau$ a {\it tail}. A double extension 
$A_P[y_1,y_2;\sigma,\delta,\tau]$ is called {\it trimmed}
if $\delta=0$ and $\tau=\{0,0,0\}$.

\begin{lemma}
\label{xxlem1.4}
Suppose $A$ and $B$ are generated in degree 1. Then $B$ is a 
$2$-extension of $A$ if and only if $B$ is a double extension 
of $A$.
\end{lemma}

\begin{proof} If $B$ is a double extension of $A$, then by
\cite[Proposition 1.14]{ZZ}, $A_{\geq 1}B=BA_{\geq 1}$,
which is denoted by $(A_{\geq 1})$, and 
$$B/(A_{\geq 1})=k\langle Y_1,Y_2\rangle/
(Y_2Y_1-p_{12}Y_1Y_2-p_{11}Y_1^2)$$ 
where $Y_i$ is the image of $y_i$ in $B/(A_{\geq 1})$.
So $C=k\langle Y_1,Y_2\rangle/
(Y_2Y_1-p_{12}Y_1Y_2-p_{11}Y_1^2)$ and it is regular of 
dimension 2. Condition (d) in the definition of $2$-extension 
follows from Definition \ref{xxdefn1.3}(aiii,biii).

Converse, we assume that $B$ is a $2$-extension. So there is
an ``exact sequence''
$$0\to A\to B\to C\to 0$$
where $C$ is regular of dimension 2. Since $C$ is generated in degree
1, $C$ is isomorphic to $k\langle Y_1,Y_2\rangle/
(Y_2Y_1-p_{12}Y_1Y_2-p_{11}Y_1^2)$ 
for some $p_{12},p_{11}\in k$ and $p_{12}\neq 0$ 
\cite[Theorem 0.2]{Zh2}.
Lifting $Y_1$ and $Y_2$ to $y_1$ and $y_2$ in $B$. 
Then $y_1$ and $y_2$ satisfy a relation
\begin{equation}
\label{E1.4.1}
y_2y_1=p_{12}y_1y_2+p_{11}y_1^2+\tau_{1}y_1+\tau_{2}y_2
+y_1\tau_{1}'+y_2\tau_{2}'+\tau_{0}
\tag{E1.4.1}
\end{equation}
for some $\tau_{i},\tau_{i}'\in A$. 
By Definition \ref{xxdefn1.1}(d), $B=A\otimes \bar{C}=\bar{C}\otimes A$
as left and right free $A$-modules respectively. Thus 
$$A_1y_1+A_1y_2\subset B_2=(\bar{C}\otimes A)_2=\bar{C}_2+y_1A_1+y_2A_2
+A_2.$$
By Definition \ref{xxdefn1.1}(c), $A_{\geq 1}B=BA_{\geq 1}$. Because
$\bar{C}_2\cap (A_{\geq 1})=\{0\}$,
$$A_1y_1+A_1y_2\subset y_1A_1+y_2A_2 +A_2.$$
Since $A$ is generated by $A_1$, by induction on the degree of 
elements in $A$, we obtain that
$$Ay_1+Ay_2\subset y_1A+y_2A+A.$$
Similarly, $y_1A+y_2A \subset Ay_1+Ay_2+A$. Consequently,
\begin{equation}
\label{E1.4.2}
Ay_1+Ay_2+A=y_1A+y_2A+A
\tag{E1.4.2}
\end{equation}
and it is a free $A$-module of rank $3$ by Definition \ref{xxdefn1.1}(d).
Using \eqref{E1.4.2}, we can re-write \eqref{E1.4.1} as
$$y_2y_1=p_{12}y_1y_2+p_{11}y_1^2+\tau_{1}y_1+\tau_{2}y_2
+\tau_{0}.$$
Therefore (R1) and Definition \ref{xxdefn1.3}(aii) holds. 
Definition \ref{xxdefn1.3}(ai) is clear and Definition 
\ref{xxdefn1.3}(aiv) is \eqref{E1.4.2}. Next we show 
Definition \ref{xxdefn1.3}(aiii). By using \eqref{E1.4.1}
and \eqref{E1.4.1}, every element in $B$ can be written as
$\sum_{n_1,n_2\geq 0}a_{n_1,n_2}y_1^{n_1}y_2^{n_2}$ for
$a_{n_1,n_2}\in A$. This implies that $B=
\sum_{n_1,n_2\geq 0}Ay_1^{n_1}y_2^{n_2}$. Since $B$ is a
2-extension, by Definition \ref{xxdefn1.1}(d), the Hilbert
series of $B$ is  $H_B(t)=H_A(t)H_C(t)$. Thus $B$ is a
left $A$-module with basis $\{y_1^{n_1}y_2^{n_2}| n_1,n_2\geq 0\}$.
So Definition Definition \ref{xxdefn1.3}(aiii) holds, and 
$B$ is a right double extension of $A$. By symmetry, 
$B$ is a left double extension of $A$ with the same 
generating set $\{y_1,y_2\}$. Therefore $B$ is a double extension.
\end{proof}

The following definition is given in \cite[Definition 1.8]{ZZ}.

\begin{definition}\cite[Definition 1.8]{ZZ}
\label{xxdefn1.5}
Let $\sigma: A\to M_2(A)$ be an algebra homomorphism.
We say $\sigma$ is {\it invertible} if there is 
an algebra homomorphism
$$\phi=\begin{pmatrix} \phi_{11}&\phi_{12}\\
\phi_{21}&\phi_{22}\end{pmatrix}: A\to M_2(A)$$
satisfies
the following conditions:
$$\sum_{k=1}^2 \phi_{jk}(\sigma_{ik}(r))=
\begin{cases} r& \text{if  } i=j\\
               0& \text{if  } i\neq j
\end{cases}
\quad \text{and}\quad
\sum_{k=1}^2 \sigma_{kj}(\phi_{ki}(r))=
\begin{cases} r& \text{if  } i=j\\
               0& \text{if  } i\neq j
\end{cases}
$$
for all $r\in A$, or equivalently,
$$\begin{pmatrix} \phi_{11}&\phi_{12}\\
\phi_{21}&\phi_{22}\end{pmatrix}
\bullet
\begin{pmatrix} \sigma_{11}&\sigma_{21}\\
\sigma_{12}&\sigma_{22}\end{pmatrix}
=
\begin{pmatrix} \sigma_{11}&\sigma_{21}\\
\sigma_{12}&\sigma_{22}\end{pmatrix}
\bullet
\begin{pmatrix} \phi_{11}&\phi_{12}\\
\phi_{21}&\phi_{22}\end{pmatrix}
=\begin{pmatrix} Id_A&0\\
0&Id_A\end{pmatrix}$$
where $\bullet$ is the multiplication of
the matrix algebra $M_2(\End_k(A))$. The
multiplication of $\End_k(A)$ is the 
composition of $k$-linear maps.
The map $\phi$ is called the {\it inverse} 
of $\sigma$. 
\end{definition}

By \cite[Lemma 1.9]{ZZ} if $B=A_P[y_1,y_2;
\sigma,\delta,\tau]$ is a double extension of
$A$, then $\sigma$ is invertible in the sense of
Definition \ref{xxdefn1.5}. As in \cite[Section 3]{ZZ}, 
one can define the determinant of $\sigma$, 
denoted by $\det \sigma$. By \cite[Section 4]{ZZ} 
$\det \sigma$ plays an essential role in the proof 
of regularity of a double extensions.

Next we will list the relations (or the constraints)
between the DE-data that come from commuting $r\in A$ 
with (R1). The collection of the following six relations 
is called (R3) for short.

\bigskip
\bigskip

\centerline{Relations (R3)}
\begin{align}
\label{R3.1}
\sigma_{21}&(\sigma_{11}(r))+p_{11}\sigma_{22}(\sigma_{11}(r))
\tag{R3.1}\\
\notag
&=p_{11}\sigma_{11}(\sigma_{11}(r))+p_{11}^2\sigma_{12}(\sigma_{11}(r))
+p_{12}\sigma_{11}(\sigma_{21}(r))+p_{11}p_{12}\sigma_{12}(\sigma_{21}(r))
\end{align}

\begin{align}
\label{R3.2}
\sigma_{21}&(\sigma_{12}(r))+p_{12}\sigma_{22}(\sigma_{11}(r))
\tag{R3.2}\\
\notag
&=
p_{11}\sigma_{11}(\sigma_{12}(r))+p_{11}p_{12}\sigma_{12}(\sigma_{11}(r))
+p_{12}\sigma_{11}(\sigma_{22}(r))+p_{12}^2\sigma_{12}(\sigma_{21}(r))
\end{align}

\begin{align}
\label{R3.3}
\sigma_{22}&(\sigma_{12}(r))
\tag{R3.3}\\
&=p_{11}\sigma_{12}(\sigma_{12}(r))+p_{12}\sigma_{12}(\sigma_{22}(r))
\qquad\qquad\qquad\qquad\qquad\qquad\qquad\qquad\quad
\notag
\end{align}

\begin{align}
\label{R3.4}
\sigma_{20}&(\sigma_{11}(r))+\sigma_{21}(\sigma_{10}(r))+\tau_{1}
\sigma_{22}(\sigma_{11}(r))\qquad\qquad\qquad\qquad\qquad\qquad
\qquad \qquad
\tag{R3.4}\\
&=p_{11}[\sigma_{10}(\sigma_{11}(r))
+\sigma_{11}(\sigma_{10}(r))
+\tau_{1}\sigma_{12}(\sigma_{11}(r))]
\notag\\
&
\quad +p_{12}[\sigma_{10}(\sigma_{21}(t))
+\sigma_{11}(\sigma_{20}(r))
+\tau_{1}\sigma_{12}(\sigma_{21}(r))]
+\tau_{1}\sigma_{11}(r)+\tau_{2}\sigma_{21}(r)
\notag
\end{align}

\begin{align}
\label{R3.5}
\sigma_{20}&(\sigma_{12}(r))
+\sigma_{22}(\sigma_{10}(r))+\tau_{2}
\sigma_{22}(\sigma_{11}(r))
\qquad\qquad\qquad\qquad\qquad\qquad\qquad
\tag{R3.5}\\
&=p_{11}[\sigma_{10}(\sigma_{12}(r))
+\sigma_{12}(\sigma_{10}(r))
+\tau_{2}\sigma_{12}(\sigma_{11}(r))]
\notag\\
&
\quad +p_{12}[\sigma_{10}(\sigma_{22}(r))+
\sigma_{12}(\sigma_{20}(r))
+\tau_{2}\sigma_{12}(\sigma_{21}(r))]
+\tau_{1}\sigma_{12}(r)+\tau_{2}\sigma_{22}(r)
\notag
\end{align}

\begin{align}
\label{R3.6}
\sigma_{20}&(\sigma_{10}(r))+
\tau_{0}\sigma_{22}(\sigma_{11}(r))
\tag{R3.6}\\
&=p_{11}[\sigma_{10}(\sigma_{10}(r))+
\tau_{0}\sigma_{12}(\sigma_{11}(r))]
\qquad\qquad\qquad\qquad\qquad\qquad\qquad \quad \quad
\notag\\
&\quad +p_{12}[\sigma_{10}(\sigma_{20}(r))+
\tau_{0}\sigma_{12}(\sigma_{21}(r))]
+\tau_{1}\sigma_{10}(r)+\tau_{2}\sigma_{20}(r)+\tau_{0}r.
\notag
\end{align}

\bigskip

The following is a combination of \cite[Propositions 1.11 and 1.13]{ZZ}.

\begin{proposition}
\label{xxprop1.6} 
Let $A$ be an algebra. 
Suppose $\{P,\sigma,\delta,\tau\}$ be a set of data such that
$\sigma: A\to M_2(A)$ is an algebra homomorphism and  
$\delta: A\to A^{\oplus 2}$ is a $\sigma$-derivation and that 
$P=\{p_{12},p_{11}\}\subseteq k$ and 
$\tau=\{\tau_{1},\tau_{2},\tau_{0}\}\subseteq A$. 
\begin{enumerate}
\item
Assume that \textup{(R3)} holds for all $r\in X$ where $X$ is a set 
of generators of $A$. Let $B$ be the algebra generated by $A$ and 
$y_1,y_2$ subject to the relations \textup{(R1)} and 
\textup{(R2)} for generators $r\in X$. Then $B$ is a right double  
extension of $A$. Namely, $B$ is a left free $A$-module with a basis
$\{y_1^{n_1}y_2^{n_2}\;|\; n_1,n_2\geq 0\}$.
\item
If further $B$ is connected graded, $p_{12}\neq 0$ and $\sigma$
is invertible, then $B$ is a double extension of $A$.
\end{enumerate}
\end{proposition}

\section{Regular algebras of dimension four}
\label{xxsec2} 

In this section we discuss some homological properties of 
(Artin-Schelter) regular algebras. We assume 
that all graded algebras in this section are generated in 
degree 1. 

The definition of regularity is recalled in Section 1.
If $B$ is regular, then by \cite[Proposition 3.1.1]{SZ}, 
the trivial left $B$-module $_Bk$ has a minimal free 
resolution of the form
\begin{equation}
\label{E2.0.1}
0\to P_{d}\to \cdots P_{1}\to P_{0}\to k_B\to 0
\tag{E2.0.1}
\end{equation}
 where $P_{w}=\oplus_{s=1}^{n_w}B(-i_{w,s})$ for some
finite integers $n_w$ and $i_{w,s}$. The Gorenstein condition 
(AS2) implies that the above free resolution is symmetric in the 
 sense that the dual complex of \eqref{E2.0.1} is a
free resolution of the trivial right $B$-module (after a 
degree shift). As a consequence,
we have $P_0=B$, $P_{d}=B(-l)$, $n_w=n_{d-w}$, and
$i_{w,s}+ i_{d-w, n_w-s+1}=l$ for all $w,s$.

Regular algebras of dimension three have been classified by 
Artin, Schelter, Tate and Van den Bergh \cite{AS,ATV1,ATV2}. 
If $B$ is a regular algebra of dimension three, then it is generated by 
either two or three elements. If $B$ is generated by three elements, 
then $B$ is Koszul and the trivial $B$-module $k$ has a minimal free 
resolution  of form
$$0\to B(-3)\to B(-2)^{\oplus 3}\to B(-1)^{\oplus 3}\to B\to k\to 0.$$
If $B$ is generated by two elements, then $B$ is not Koszul and the 
trivial $B$-module $k$ has a minimal free resolution of the form
$$0\to B(-4)\to B(-3)^{\oplus 2}\to B(-1)^{\oplus 2}\to B\to k\to 0.$$

If $B$ is a noetherian regular algebra of (global) dimension four, 
then $B$ is generated by either 2, or 3 or 4 elements
\cite[Proposition 1.4]{LPWZ}. Minimal free resolutions of the 
trivial module $k$ is listed in \cite[Proposition 1.4]{LPWZ}. 
The following lemma is well-known. The transpose of a matrix $M$ 
is denoted by $M^T$.

\begin{lemma}
\label{xxlem2.1} 
Let $B$ be a regular graded domain of dimension four. 
Suppose $B$ is generated by elements $x_1,x_2,x_3,x_4$ (of degree 1). 
\begin{enumerate}
\item
$B$ is of type (14641), namely, the trivial left $B$-module $k$ has 
a free resolution
\begin{equation}
\label{E2.1.1}
0\to B(-4) \xrightarrow{\partial_4} B^{\oplus 4}(-3)
\xrightarrow{\partial_3} B^{\oplus 6}(-2)
\xrightarrow{\partial_2} B^{\oplus 4}(-1)
\xrightarrow{\partial_1} B
\xrightarrow{\partial_0} k\to 0
\tag{E2.1.1}
\end{equation}
where $B^{\oplus n}$ is the free left $B$-module written as an 
$1\times n$ matrix.
\item
$\partial_0$ is the augmentation map with
$\ker \partial_0=B_{\geq 1}$.
\item
$\partial_1$ is given by the right multiplication by
$(x_1,x_2,x_3,x_4)^T$.
\item
$\partial_2$ is the right multiplication by a $6\times
4$-matrix
$F=(f_{ij})_{6\otimes 4}$ such that $f_i:=\sum_{j=1}^4
f_{ij} x_j$,
for $i=1,2,3,4,5,6$, are the 6 relations of $B$.
\item
$\partial_3$ is the right multiplication by a $4\times
6$-matrix
$G=(g_{ij})_{4\times 6}$.
\item
$\partial_4$ is the right multiplication by
$(x'_1,x'_2,x'_3,x'_4)$
where $\{x'_1,x'_2,x'_3,x'_4\}$ is a set of generators
of $B$.
\textup{(}So each $x'_i$ is a $k$-linear combination
of $\{x_i\}_{i=1}^4$.
\textup{)}
\item 
$F(x_1,x_2,x_3,x_4)^T=0$, $GF=0$,
$(x'_1,x'_2,x'_3,x'_4)G=0$.
\end{enumerate}
\end{lemma}

The dual complex of \eqref{E2.1.1} is obtained by 
applying the functor $(-)^\vee:=\Hom_B(-,B)$ to 
\eqref{E2.1.1}. Condition (AS2) implies that 
the dual complex of \eqref{E2.1.1} is a free 
resolution of the right $B$-module $k(4)$:
$$0\leftarrow k_B(4) \leftarrow B(4)
\xleftarrow{\partial_4^\vee} 
B^{\oplus 4}(3)
\xleftarrow{\partial_3^\vee} B^{\oplus 6}(2)
\xleftarrow{\partial_2^\vee} B^{\oplus 4}(1)
\xleftarrow{\partial_1^\vee} B
\leftarrow 0.$$
Lemma \ref{xxlem2.1}(f) follows from this 
observation. Other parts of Lemma 
\ref{xxlem2.1} are clear. 

\begin{lemma}
\label{xxlem2.2}
Let $B$ be as in Lemma \ref{xxlem2.1}.
\begin{enumerate}
\item
Each column and each row of $F$ and $G$ is nonzero.
\item
If $\alpha$ is a nonzero row vector in $k^4$
and $\beta$ is a nonzero row  vector in $k^6$, then
$F\alpha^T\neq 0$, $\alpha G\neq 0$, $\beta F\neq 0$ and 
$G\beta^T\neq 0$.
\item
The subspace spanned by elements in a fix column (or
row) of either $F$ or $G$ has dimension at least 2.
\end{enumerate}
\end{lemma}

\begin{proof} (a) If a row of $F$ is zero, then $\ker
\partial_2$
contains a copy of $B(-2)$ and \eqref{E2.1.1} is not
exact;
that is a contradiction. So any row of $F$ is nonzero.
Suppose now 
a column of $F$ is zero. We consider the dual complex
of \eqref{E2.1.1}:
\begin{equation}
\label{E2.2.1}
0\leftarrow k_B(4) \leftarrow B(4)
\xleftarrow{\partial_4^\vee} 
B^{\oplus 4}(3)
\xleftarrow{\partial_3^\vee} B^{\oplus 6}(2)
\xleftarrow{\partial_2^\vee} B^{\oplus 4}(1)
\xleftarrow{\partial_1^\vee} B
\leftarrow 0
\tag{E2.2.1}
\end{equation}
where each $B^{\oplus n}$ is an $n$-column right free
$B$-module.
This complex is a free resolution of $k_B(4)$ by the
Gorenstein
condition (AS2). The map $\partial_2^\vee$ is the left
multiplication 
by $F^T$. If some column of $F$ is zero, then some
row of $F^T$
is zero and $\ker \partial_2^\vee$ contains some a
copy of 
$B(1)$. So complex \eqref{E2.2.1} is not exact 
at $B^{\oplus 4}(1)$, a contradiction. 

The same proof works for $G$.

(b) Let $M$ be a $4\times 4$ non-singular matrix such
that $\alpha^T$
is the first column of $M$. Replace
$X:=\{x_1,x_2,x_3,x_4\}^T$ by 
another generating set $X':=M X$ will change $F$ to
$F':=FM$. The 
first column of $F'$ is zero if $F\alpha^T=0$.
This contradicts
with part (a). So $F\alpha^T\neq 0$. Similarly,
$G\beta^T\neq 0$.
For $\beta F\neq 0$ and $\alpha G\neq 0$ we use the
dual
complex of \eqref{E2.1.1}.

(c) If the dimension of the subspace spanned by the first
column of $F$ is 1 (which can not be zero by part (b)), 
then $f_{i1}=\beta_i v$ for some
$\beta_i\in k$ and $0\neq v\in B_1$. Then $G\beta^T
v=0$ for 
$\beta=(\beta_1, \cdots,\beta_6)$. Since $B$ is a
domain, we have 
$G\beta^T=0$, which contradicts with part (b). 

If the dimension of the spanned by the first column of
$G$ is 1,
then $g_{i1}=\alpha_i v$ for some $\alpha_i\in k$ and
$0\neq v\in B_1$. 
Then $(x'_1,x'_2,x'_3,,x'_4)\alpha^T v=0$. Since
$B$ is a 
domain $(x'_1,x'_2,x'_3,,x'_4)\alpha^T=0$. This
implies that
$x'_1,x'_2,x'_3,,x'_4$ are $k$-linearly dependent, a
contradiction.

By using the dual complex of \eqref{E2.1.1} we can
prove the assertion
for the rows of $F$ and $G$. 
\end{proof}

Lemma \ref{xxlem2.2} can be used to show some graded 
algebras are not regular. Here is an example.

\begin{proposition}
\label{xxprop2.3} Let $B$ be a graded domain generated by
elements $x_1,x_2,x_3,x_4$. Suppose $B$ has 
the 6 quadratic relations of following form:
\begin{align}
x_1x_4&=qx_4x_1, q\in k;\notag\\
x_4^2&=f(x_1,x_2,x_3)\neq g(x_1,x_2,x_3) x_1;\notag\\
{\text {and }} 4 &\text{    other relations only involving
}x_1,x_2,x_3.\notag
\end{align}
Then $B$ is not regular of dimension four.
\end{proposition}

\begin{proof} Suppose on the contrary that $B$ is regular of 
dimension four. Then we can use Lemma \ref{xxlem2.2} and 
use the notations introduced there. The form of the relations 
implies that
$$F=\begin{pmatrix}
qx_4& 0 & 0 & x_1\\
*   & * & * & x_4\\
*   & * & * & 0  \\
*   & * & * & 0  \\
*   & * & * & 0  \\
*   & * & * & 0  
\end{pmatrix}$$
where all $*$ are in $kx_1+kx_2+kx_3$. Since $GF=0$,
the 4th column of this matrix equation gives
$g_{i1}x_1+g_{i2}x_4=0$ for all $i=1,2,3,4$. Looking
back at the 6 relations of $B$, one sees that only
the first two relations have the term $gx_4$. This
implies that $g_{i2}=a_ix_1+b_ix_4$ for 
some $a_i,b_i\in k$. But by the first two relations
$$(a_ix_1+b_ix_4)x_4=a_iqx_4x_1+b_if(x_1,x_2,x_3)$$
where the right-hand side is not of the form $g'x_1$ 
unless $b_i=0$ for all $i$. Therefore 
$g_{i1}x_1+g_{i2}x_4=0$ implies that $b_i=0$ for all $i$. 
Now the space spanned by the second row of $G$ is $kx_1$,
which is 1-dimensional. We obtain a contradiction by
Lemma \ref{xxlem2.2}(c). Therefore $B$ is not regular of
dimension four. 
\end{proof}

Next we want to show that if a right double extension
is a regular domain of type (14641), then it is automatic
a double extension.

\begin{lemma}
\label{xxlem2.4}
Let $A$ be a connected graded algebra and let 
$B=A_P[y_1,y_2;\sigma,\delta,\tau]$ be a right double
extension of $A$. If $B$ is a regular domain of type 
(14641). Then $A$ is a regular domain of dimension 2 
with Hilbert series $(1-t)^{-2}$, namely, $A$ is 
isomorphic to either $k_{p}[x_1,x_2]$ or $k_{J}[x_1,x_2]$.  
\end{lemma}

\begin{proof} Since $H_B(t)=(1-t)^{-4}$ and $B$ is a
free $A$-module with basis 
$\{y_1^{n_1}y_2^{n_2}\}_{n_1,n_2\geq 0}$,
we see that $H_A(t)=(1-t)^{-2}$.

Let $A'$ be the subalgebra of $A$ 
generated by the elements of degree 1. Then $A'$ is a 
domain and generated by two elements of degree 1, say
$x_1,x_2$. Since $H_A(t)=(1-t)^{-2}$, $A$ (and hence 
$A'$) has at least one relation in degree 2. Any graded 
domain generated by two elements with at least one 
relation in degree 2 is isomorphic to $k\langle
x_1,x_2\rangle/(x_2x_1-p_{12}x_1x_2-p_{11}x_1^2)$
for some $p_{ij}\in k$ with $p_{12}\neq 0$ (this 
is well-known and a proof is given in \cite[Lemma 3.7]{KKZ}). 
Therefore $A'=k\langle x_1,x_2\rangle/(x_2x_1-p_{12}x_1x_2
-p_{11}x_1^2)$ and hence $H_{A'}(t)=(1-t)^{-2}$.
Thus $A$ and $A'$ has the same Hilbert series, whence 
$A=A'$. In particular, $A$ (and hence $B$) has a
relation
$$x_2x_1=p_{12}x_1x_2+p_{11}x_1^2.$$
After a linear transformation $(p_{12},p_{11})$ can be
chosen to be either $(p,0)$ or $(1,1)$ that corresponds
to regular algebras $k_{p}[x_1,x_2]$ and $k_{J}[x_1,x_2]$ 
respectively.
\end{proof}

Let $V$ denote the vector space $kx_1+kx_2$ and $W$ denote 
the vector space $V+ky_1+ky_2$. By Lemma \ref{xxlem2.4},
$B$ contains a relation in $V\otimes V$. Recall that a 
right double extension $B$ has a relation (R1):
$$y_2y_1=p_{12}y_1y_2+p_{11}y_1^2+\tau_{1}y_1+\tau_{2}y_2
+\tau_{0}.$$
The condition (R2) implies that there are 4 relations
in $V\otimes W+W\otimes V$.

\begin{lemma}
\label{xxlem2.5}
Let $B$ be a regular domain of type (14641).
Suppose that $B$ is generated by $x_1,x_2,y_1,y_2$ 
satisfying the following quadratic relations:
\begin{enumerate}
\item[(i)]
$f_1=0$ for some $f_1\in V\otimes V$.
\item[(ii)] 
four relations $f_i=0$ for $i=2,3,4,5$ where 
$f_i\in V\otimes W+W\otimes V$.
\item[(iii)]
one relation $f_6$ of the modified form of 
\textup{(R1)} 
\begin{equation}
\label{E2.5.1}
y_2y_1=p_{12}y_1y_2+p_{11}y_1^2+\tau_{1}y_1+\tau_{2}y_2
+y_1\tau_{1}'+y_2\tau_{2}'+\tau_{0}
\tag{E2.5.1}
\end{equation}
with $p_{12},p_{11}\in k$ and $0\neq p_{12}$, and where 
$\tau_{1},\tau_{2},\tau_{1}',\tau_{2}'\in V$ and 
$\tau_{0}\in V\otimes V$.
\end{enumerate}
Then
\begin{enumerate}
\item
$VW=WV$ in $B$.
\item
Let $A$ be the subalgebra of $B$ generated by $V$.
Then $Ay_1+Ay_2+A=y_1A+y_2A+A$.
\item
$B$ is a double extension $A_P[y_1,y_2;\sigma,\delta,\tau]$.
In particular, $\sigma$ is invertible.
\end{enumerate}
\end{lemma}

\begin{proof} (a) The assertion is equivalent to
equation 
$$Vy_1+Vy_2+V^2=y_1V+y_2 V+V^2.$$ 
Suppose this is 
not true. Then we have the following two cases: either
$$Vy_1+Vy_2+V^2\subsetneq y_1V+y_2 V+V^2,$$ 
or 
$$Vy_1+Vy_2+V^2\supsetneq y_1V+y_2 V+V^2.$$ 
Note that the relation \eqref{E2.5.1} is left-right symmetric
and that all other five relations are clearly left-right 
symmetric. By symmetry let us only consider the first case.
Let $V_0$ be the maximal subspace of $V$ such that 
$$V_0y_1+V_0y_2\subset WV=y_1V+y_2 V+V^2.$$
By the assumption of the first case, $\dim V_0\leq 1$. 
Write the six relations of $B$ as 
$$0=f_i=f_{i1}x_1+f_{i2}x_2+f_{i3}y_1+f_{i4}y_4$$
for $i=1,\cdots,6$. 
Let $F=(f_{ij})_{6\times 4}$ be the matrix defined as 
in Lemma \ref{xxlem2.1}(d).
The relation \eqref{E2.5.1} can be written as 
$$0=f_6=f_{61}x_1+f_{62}x_2+f_{63}y_1+f_{64}y_2$$
where $f_{63}=-y_2+p_{11}y_1+\tau_{1}$ and
$f_{64}=p_{12}y_1+\tau_{2}$.
Since all other five relations are elements 
in $W\otimes V+V\otimes W$, we have $f_{i3},
f_{i4}\in V$ for all $i=1,2,3,4,5$.
Let $G=(g_{ij})_{4\times 6}$ be defined as in 
Lemma \ref{xxlem2.1}(e).
Since $GF=0$ [Lemma \ref{xxlem2.1}(g)], the third 
and fourth columns of this
matrix equation implies that, for all $i=1,2,3,4$,
\begin{equation}
\label{E2.5.2}
\qquad
g_{i6}(p_{12}y_1+\tau_{2})=-\sum_{k=1}^5
g_{ik}f_{k3}\in WV
\tag{E2.5.2}
\end{equation}
\begin{equation}
\label{E2.5.3}
g_{i6}(-y_2+p_{11}y_1+\tau_{1})=-\sum_{k=1}^5
g_{ik}f_{k4}\in WV.\quad
\tag{E2.5.3}
\end{equation}
Since $p_{12}\neq 0$ and $\tau_{2}\in V$, \eqref{E2.5.2} implies 
that $g_{i6}y_1\in WV$. Since \eqref{E2.5.2} does not contain 
the term $y_1y_2$ and since \eqref{E2.5.1} has a nonzero term 
$p_{12}y_1y_2$, \eqref{E2.5.2} is a linear combination of relations
other than \eqref{E2.5.1}. So $g_{i6}\in V$. A linear combination
of \eqref{E2.5.2} and \eqref{E2.5.3} shows that $g_{i6}y_2\in 
WV$. By the definition of $V_0$,  $g_{i6}\in
V_0$ for all $i$, and hence the space spanned 
by the sixth column of $G$ has dimension
at most 1. This contradicts Lemma \ref{xxlem2.2}(c).
Therefore part (a) follows. 

(b) By part (a) and induction one sees that $V^nW=WV^n$
for all $n$. Hence $\sum_{n\geq 0} V^nW=\sum_{n\geq 0}
WV^n$. Since $A=\sum_{n\geq 0} V^n$, we have $AW=WA$, or 
$$Ay_1+Ay_2+A_{\geq 1}=y_1A+y_2A+A_{\geq 1}.$$ 
The assertion follows by adding $k$ to the above equation. 

(c) By Lemma \ref{xxlem2.4}, $A$ is a regular algebra of 
dimension 2 with Hilbert series $H_A(t)=(1-t)^{-2}$.

By part (a), the relation \eqref{E2.5.1} can 
be simplified to the form of (R1), namely
we may assume $\tau_{1}'=\tau_{2}'=0$. We claim that 

{\it every element $f$ in $B$ can be written as an 
element in $\sum_{n_1,n_2\geq 0} Ay_1^{n_1} y_2^{n_2}$.}

Let $\deg_y$ be the degree of an element with respect
to $y_1$ and $y_2$, namely, $\deg_y x_i=0$ and $\deg_y y_i=1$
for $i=1,2$. If $\deg_y f\leq 1$, the assertion follows from part (b). 
If $\deg_y f>1$, the assertion follows from the induction on 
$\deg_y$, part (b) and the relation (R1) (which is 
equivalent to \eqref{E2.5.1} after we proved part (b)). 

Since the Hilbert series of $B$ is $(1-t)^{-4}$, an easy 
computation shows that $\sum_{n_1,n_2\geq 0} Ay_1^{n_1}
y_2^{n_2}$ is a free $A$-module with basis 
$\{y_1^{n_1} y_2^{n_2}\}_{\{n_1,n_2\geq 0\}}$. By
Definition \ref{xxdefn1.3}, $B$ is a right double extension 
of $A$; so we can write $B=A_P[y_1,y_2;\sigma,\delta,\tau]$. 
By part (b) we have $Ay_1\oplus Ay_2\oplus A=y_1 A+y_2 A+ A$.
This implies that the Hilbert series of $y_1 A+y_2 A+ A$
is equal to the Hilbert series of $Ay_1\oplus Ay_2\oplus A$,
which is $(1+2t)(1-t)^{-2}$. By the $k$-dimensional counting,
$y_1 A+y_2 A+ A$ must be free over $A$ with basis $\{1,y_1,y_2\}$.
Hence 
$$Ay_1\oplus Ay_2\oplus A=Ay_1+ Ay_2+ A=y_1 A+y_2 A+ A=
y_1 A\oplus y_2 A\oplus A.$$ 
By \cite[Lemma 1.9]{ZZ}, $\sigma$ is invertible. Since
$p_{12}\neq 0$, by \cite[Proposition 1.13]{ZZ}, $B$ is
a double extension.
\end{proof}

In Lemma \ref{xxlem2.5} we assumed that $p_{12}\neq 0$. 
We show next that this condition is not too restrictive.

\begin{lemma} 
\label{xxlem2.6}
Let $B$ be a regular domain of type (14641) generated 
by $x_1,x_2,y_1,y_2$. Suppose $B$ has 
six quadratic relations satisfying the following conditions:
\begin{enumerate}
\item[(i)]
The first relation is $f_1=0$ for some $f_1\in V\otimes V$.
\item[(ii)]
There are four relations in $V\otimes W+W\otimes V$ such that
\begin{enumerate}
\item[(iia)] 
two of these are of the form $f_2'=0, f_3'=0$ 
where $f_2'=y_2x_1-h_2$ and 
$f_3'=y_2x_2-h_3$ for some $h_2,h_3\in V\otimes W$;
\item[(iib)]
and other two are $f_i'=0$ for $i=4,5$ where 
$f_i\in V\otimes W+W\otimes V$.
\end{enumerate}
\item[(iii)]
The last relation $f_6$ is of the form 
$$-y_2y_1+p_{12}y_1y_2+p_{11}y_1^2+h=0$$
for some $h\in W\otimes V+V\otimes W$.
\end{enumerate}
Then $B$ can not have three relations $f_i=0$ for
linearly independent elements $\{f_1,f_2,f_3\}
\subset W\otimes V$.
\end{lemma}

\begin{proof} If $p_{12}\neq 0$, by Lemma \ref{xxlem2.5},
$B$ is a double extension and $\sigma$ is invertible. In this 
case, it is easy to see that there is only one relation 
$f_1\in W\otimes V$.

In the rest of the proof, we assume $p_{12}=0$. We continue 
to use the notations introduced in Lemma \ref{xxlem2.1}.

Assume on the contrary that three of six relations of $B$ are 
of the form $f_i=0$ where $f_i\in W\otimes V$ for $i=1,2,3$. 
This implies that the matrix $F$ is of the form
$$\begin{pmatrix} \bullet  & \bullet        & 0 & 0 \\
\bullet  & \bullet        & 0 & 0 \\
\bullet  & \bullet        & 0 & 0 \\
\bullet  & \bullet        & f_{43} & *\\
\bullet  & \bullet        & f_{53} & * \\
\bullet  & \bullet        & -y_2+p_{11}y_1+q_3 & * 
\end{pmatrix}$$
where $\bullet$ denotes elements in $W$, and $*$ denotes elements 
in $V$, and $f_{43},f_{53},q_3\in V$. The third column of the 
matrix equation $GF=0$ gives rise to the equations
\begin{equation}
\label{E2.6.1}
g_{i4}f_{43}+g_{i5}f_{53}+
g_{i6}(-y_2+p_{11}y_1+q_3)=0
\tag{E2.6.1}
\end{equation}
for $i=1,2,3,4$. Since we assume $p_{12}=0$, any quadratic relation
of $B$ contains neither $y_1y_2$ nor $y_2^2$. This implies that
$g_{i6}\in V$ for all $i$. By Lemma \ref{xxlem2.2}(c),
the vector space $V'$ spanned by $g_{i6}$ for
$i=1,\cdots,4$ has dimension at least 2, whence 
$V'=V$. This means that there are at least two 
relations which can be derived from \eqref{E2.6.1}
with $g_{i6}\neq 0$. Up to linear transformation we may assume 
that the relations derived from \eqref{E2.6.1} are of the form
$$h_4+x_1(-y_2+p_{11}y_1+q_3)=0\quad\text{and}\quad
h_5+x_2(-y_2+p_{11}y_1+q_3)=0$$
where $h_4,h_5\in W\otimes V$. Denote these two relations as 
$f_4=0$ and $f_5=0$. Clearly, $\{f_4,f_5\}$ is linearly 
independent in the quotient space $(V\otimes W+W\otimes V)
/W\otimes V$. Recall that, for the first three 
relations $f_i=0$, $i=1,2,3$, we have $f_i\in W\otimes V$ and, 
for the sixth relation, $f_6$ contains a nonzero monomial 
$y_2y_1$. Hence $\{f_1,f_2,f_3,f_4,f_5,f_6\}$ are linearly 
independent, and we may use these relations as the defining 
relations of $B$. Hence the matrix $F$ becomes
$$\begin{pmatrix} \bullet  & \bullet        & 0 & 0 \\
\bullet  & \bullet        & 0 & 0 \\
\bullet  & \bullet        & 0 & 0 \\
\bullet  & \bullet        & p_{11}x_1 & -x_1\\
\bullet  & \bullet        & p_{11}x_2 & -x_2 \\
\bullet  & \bullet        & -y_2+p_{11}y_1+q_3 & f_{64} 
\end{pmatrix}$$
where $f_{64}\in V$. The third and fourth columns 
of the matrix equation $GF=0$ read as follows:
$$g_{i4}p_{11}x_1+g_{i5}p_{11}x_2+
g_{i6}(-y_2+p_{11}y_1+q_3)=0$$
and
$$g_{i4}(-x_1)+g_{i5}(-x_2)+
g_{i6}f_{64}=0.$$
Combining these two we obtain
$$g_{i6}(-y_2+p_{11}y_1+q_3+p_{11}f_{64})=0$$
which contradicts to the fact $B$ is a domain. 
Therefore we have proved that if $p_{12}=0$
then $B$ can not have three relations 
$f_1=0,f_2=0,f_3=0$ with linearly independent 
elements $\{f_1,f_2,f_3\}\subset W\otimes V$. 
\end{proof}

\begin{lemma}
\label{xxlem2.7}
Let $B$ be as in Lemma \ref{xxlem2.6}. Then 
\begin{enumerate}
\item
$p_{12}\neq 0$.
\item
$B$ is a double extension $A_{P}[y_1,y_2;\sigma,\delta,\tau]$ 
where $A$ is the subalgebra generated by $x_1$ and $x_2$.
\end{enumerate}
\end{lemma}

\begin{proof}
(a) Suppose on the contrary that $p_{12}=0$. We use the 
form of six relations given in (i), (ii) and (iii) of
Lemma \ref{xxlem2.6}.

Without loss of generality, we 
may assume that the terms $y_2x_1,y_2x_2$ do not appear in 
$f_i$ for all $i\neq 2,3$. Thus we can write the $F$ as
follows:
$$F=\begin{pmatrix} *        & *        & 0 & 0 \\
                    y_2+ q_1 & *        & * & * \\
                    *        & y_2 +q_2 & * & * \\
                    *        & *        & * & * \\
                    *        & *        & * & * \\
                    *        & *    & -y_2+p_{11}y_1+q_3& *
\end{pmatrix}$$
where $q_i\in V$ and where all $*$ denotes elements in $V+ky_1$.
For each $i$, we consider the equation
$$0=\sum_{k=1}^6 g_{ik}f_{k3}=
\sum_{k=1}^5 g_{ik} f_{k3}+g_{i6}f_{63}$$ 
which comes from the the third column of
the matrix equation $GF=0$. Since $y_1y_2$ 
and $y_2^2$ do not appear in any of the 
relations and since $f_{k3}$ does not 
contain $y_2$ for all $k\neq 6$, $g_{i6}$ 
does not contain either $y_1$ or $y_2$. 
So $g_{i6}\in V$ for all $i$. Similarly 
one can show that $g_{i2},g_{i3}\in V$ for 
all $i$. The six relations of $B$ can 
also be obtained by the multiplication 
$(x'_1,x'_2,x'_3,x'_4)G$.
So three relations corresponding to 
columns $2,3,6$ of $(x'_1,x'_2,x'_3,x'_4)G$ 
is in $W\otimes V$. But this contradicts Lemma \ref{xxlem2.6}.
Therefore $p_{12}\neq 0$.

(b,c) Follows from part (a) and Lemma \ref{xxlem2.5}(c).
\end{proof}

Now we can prove the main result in this section.

\begin{theorem}
\label{xxthm2.8}
Let $B$ be a regular domain of type (14641). Suppose one of the
following conditions holds:
\begin{enumerate}
\item[(i)]
$B$ is a right double extension.
\item[(ii)]
there are $x_1,x_2\in B_1$ such that 
\begin{enumerate}
\item[(iia)] 
$B$ has a quadratic relation involving only $x_1,x_2$, and 
\item[(iib)] 
$B/(x_1,x_2)$ is regular of dimension 2.
\end{enumerate}
\end{enumerate}
Then $B$ is a double extension of a regular subalgebra of dimension
2. 
\end{theorem}

\begin{proof} (a) By Lemma \ref{xxlem2.4} $A$ is regular of dimension
2. Since $B$ is a right double extension, hypotheses (i,iia,iib,iii)
of Lemma \ref{xxlem2.6} hold. The assertion follows from 
Lemma \ref{xxlem2.7}.

(b) We would like to check (i,ii,iii) of Lemma \ref{xxlem2.5}.
(i) is clear. (iii) follows from the fact that $B/(x_1,x_2)$
is regular of dimension 2. Since $B/(x_1,x_2)$ has only
one relation, all other relations are $f_i=0$ with $f_i\in 
V\otimes W+W\otimes V$. Thus we verified (ii) of 
Lemma \ref{xxlem2.5}. By Lemma \ref{xxlem2.5}, $B$ is a 
double extension. 
\end{proof} 

\begin{proof}[Proof of Proposition \ref{xxprop0.3}]
If $B$ is a double extension of $A$, by 
\cite[Proposition 1.14]{ZZ}, there is an algebra
homomorphism $B\to B/(A_{\geq 1})$ and $B/(A_{\geq 1})$ is isomorphic 
to $k\langle y_1,y_2\rangle/(y_2y_1-p_{12}y_1y_2-p_{11}y_1^2)$.
Since $p_{12}\neq 0$, $B/(A_{\geq 1})$ is regular of
dimension 2. This is one implication. The other implication 
is Theorem \ref{xxthm2.8}.
\end{proof}

To conclude this section we prove Theorem \ref{xxthm0.2}.

\begin{theorem}
\label{xxthm2.9} Suppose that $B$ is a noetherian regular algebra of 
type (14641) and that $B$ has a ${\mathbb Z}^2$-grading such that 
$B_1=B_{01}\oplus B_{10}$ with both $B_{01}$ and
$B_{10}$ nonzero. Then $B$ is either a double extension
or an Ore extension $A[x;\sigma]$ for some regular algebra 
$A$ of dimension three.
\end{theorem}

\begin{proof} By \cite[Theorem 3.9]{ATV2}, $B$ is a domain. 
Hence it is a quantum polynomial ring in the sense of 
\cite[Definition 1.12]{KKZ}.

By \cite[Proposition 3.5]{KKZ}, both subalgebras 
$B_{{\mathbb Z}\otimes 0}$ and $B_{0\times {\mathbb Z}}$ are 
Koszul noetherian regular domains of dimension strictly smaller 
than four. If $\dim B_{10}=3$ and $\dim B_{01}=1$, then 
$A:=B_{{\mathbb Z}\otimes 0}$ is regular of dimension three and 
$C:=B_{0\times {\mathbb Z}}$ is regular of dimension 1. 
It is straightforward to show that there is an exact sequence 
of algebras
$$0\to A\to B\to C\to 0$$
satisfying Definition \ref{xxdefn1.1}(a,b,c). Thus $B$ is
$1$-extension of $A$. By Lemma \ref{xxlem1.2}(c), 
$B=A[x;\sigma,\delta]$. Since $B$ is ${\mathbb Z}^2$-graded, 
$\delta=0$.

If $\dim B_{10}=1$ and $\dim B_{01}=3$, a similar proof works.

It remains to consider $\dim B_{10}=\dim B_{01}=2$. By
By \cite[Proposition 3.5]{KKZ}, both $B_{{\mathbb Z}\otimes 0}$ 
and $B_{0\times {\mathbb Z}}$ are Koszul regular algebras of 
dimension 2 and $B_{{\mathbb Z}\otimes 0}\cong
B/((B_{0\times {\mathbb Z}})_{\geq 1})$. So we are in the 
situation of Theorem \ref{xxthm2.8}(ii) and hence $B$ is a 
double extension.
\end{proof}

\section{System C}
\label{xxsec3}

The goal of Sections 3 and 4 is to 
classify regular domains of dimension four of 
the form $A_{P}[y_1,y_2;\sigma]$ up to isomorphism, or 
equivalently, to classify $(P,\sigma)$ up to some equivalence 
relation. This is the first step toward
a more complete (but not finished) classification of 
$A_{P}[y_1,y_2;\sigma,\delta,\tau]$. As explained in the 
introduction, we are interested in double extensions that are 
not iterated Ore extensions. 

Since the base field $k$ is algebraically closed, $A$ is 
isomorphic to $k_Q[x_1,x_2]$ (see Lemma \ref{xxlem2.4}), which 
is either $k_q[x_1,x_2]$ with the relation $x_2x_1=q x_1x_2$ 
(in this case $Q=(q,0)$) or $k_{J}[x_1,x_2]$ with the relation 
$x_2x_1=x_1x_2+x_1^2$ (in this case $Q=J=(1,1)$). To state some
results uniformly we write $A$ as $k_Q[x_1,x_2]$ where, by 
definition, $Q=(q_{12},q_{11})$ and $k_Q[x_1,x_2]=k\langle
x_1,x_2\rangle/(x_2x_1=q_{11}x_1^2+q_{12}x_1x_2)$. But in the 
computation we will set $Q$ to be either $(1,1)$ or $(0,q)$.

Fix an $A$ as in the last paragraph. Let $\sigma: A\to M_2(A)$ 
be a graded algebra homomorphism. Write
\begin{equation}
\label{E3.0.1}
\sigma_{ij}(x_s)=\sum_{s=1}^2 a_{ijst} x_t
\tag{E3.0.1}
\end{equation}
for all $i,j,s=1,2$ and where $a_{ijst}\in k$.

Using \eqref{E3.0.1} we can re-write the relation \eqref{R2} 
of the algebra $A_P[y_1,y_2;\sigma]$ as follows (note that
in this case $\delta=0$). Setting $r=x_1$ and $x_2$ in \eqref{R2}, 
we have the following four relations.

\begin{align}
\label{MR11}
y_1x_1&=\sigma_{11}(x_1)y_1+\sigma_{12}(x_1)y_2\tag{MR11}\\     
&=a_{1111}x_1y_1+a_{1112}x_2y_1+a_{1211}x_1y_2+a_{1212}x_2y_2
\notag
\end{align}
\begin{align}
\label{MR12}
y_1x_2&=\sigma_{11}(x_2)y_1+\sigma_{12}(x_2)y_2\tag{MR12}\\     
&=a_{1121}x_1y_1+a_{1122}x_2y_1+a_{1221}x_1y_2+a_{1222}x_2y_2
\notag
\end{align}
\begin{align}
\label{MR21}
y_2x_1&=\sigma_{21}(x_1)y_1+\sigma_{22}(x_1)y_2\tag{MR21}
\\
&=a_{2111}x_1y_1+a_{2112}x_2y_1+a_{2211}x_1y_2+a_{2212}x_2y_2
\notag
\end{align}
\begin{align}
\label{MR22}
y_2x_2&=\sigma_{21}(x_2)y_1+\sigma_{22}(x_2)y_2\tag{MR22}\\
&=a_{2121}x_1y_1+a_{2122}x_2y_1+a_{2221}x_1y_2+a_{2222}x_2y_2
\notag
\end{align}
We call the above 4 relations {\it mixing relations}
between $x_i$ and $y_i$. The double extension 
$A_P[y_1,y_2;\sigma]$ also has two non-mixing relations:
\begin{align}
\label{NRx}
x_2x_1&=q_{12}x_1x_2+q_{11}x_1^2
\tag{NRx}\\
\label{NRy}
y_2y_1&=p_{12}y_1y_2+p_{11}y_1^2
\tag{NRy}
\end{align}

Let $\Sigma_{ij}$ be the matrix 
$$\begin{pmatrix}
a_{ij11}&a_{ij12}\\a_{ij21}&a_{ij22}\end{pmatrix}$$
and let
$$\Sigma=\begin{pmatrix}
\Sigma_{11}&\Sigma_{12}\\\Sigma_{21}&\Sigma_{22}\end{pmatrix}
=\begin{pmatrix}
a_{1111}&a_{1112}&a_{1211}&a_{1212}\\
a_{1121}&a_{1122}&a_{1221}&a_{1222}\\
a_{2111}&a_{2112}&a_{2211}&a_{2212}\\
a_{2121}&a_{2122}&a_{2221}&a_{2222}
\end{pmatrix}.
$$
One easily sees that $\sigma$ is invertible if and
only if the
matrix $\Sigma$ is invertible. Since we assume that
$\sigma$ is
a graded algebra homomorphism, $\sigma$ is uniquely
determined
by $\Sigma$. 
Another matrix closed related to $\Sigma$ is the
following.
Let $M_{ij}$ be the matrix
$$\begin{pmatrix}
a_{11ij}&a_{12ij}\\a_{21ij}&a_{22ij}\end{pmatrix}$$
and let
$$M=
\begin{pmatrix}
M_{11}&M_{12}\\M_{21}&M_{22}\end{pmatrix}.
$$
The matrix $M$ is obtained by re-arranging the entries of 
$\Sigma$. Sometimes it is convenient to use $M$ instead of 
$\Sigma$ when we make linear  transformation of 
$\{y_1,y_2\}$. An easy linear algebra exercise
shows that $\Sigma$ is invertible if and only if $M$
is invertible.

If we change the basis $\{x_1,x_2\}$ to
$x'_1=b_{11}x_1+b_{12}x_1,
x'_2= b_{21}x_1+b_{22}x_2$, then the matrix $\Sigma$
is changed to
a new $\Sigma$. The following lemma is clear. 

\begin{lemma}
\label{xxlem3.1} Let $X=\begin{pmatrix} x_1\\x_2\end{pmatrix}$ 
and $Y=\begin{pmatrix} y_1\\y_2\end{pmatrix}$.
\begin{enumerate}
\item
The 4 mixing relations
\textup{\eqref{MR11}-\eqref{MR22}} can be 
written as
$$y_iX=y_i\begin{pmatrix} x_1\\x_2\end{pmatrix} 
=\Sigma_{i1}\begin{pmatrix}
x_1y_1\\x_2y_1\end{pmatrix}  
+\Sigma_{i2}\begin{pmatrix}
x_1y_2\\x_2y_2\end{pmatrix}=
\Sigma_{i1}Xy_1
+\Sigma_{i2}Xy_2$$
for $i=1,2$.
\item
If $X$ is changed to $X'=
BX$ where $B=(b_{ij})_{2\times 2}$ is an invertible
matrix,
then $\Sigma'$ is equal to $\begin{pmatrix}
B&0\\0&B\end{pmatrix} \Sigma \begin{pmatrix}
B^{-1}&0\\0&B^{-1}\end{pmatrix}$.
\item
Suppose $Q=(1,0)$ (in this case the algebra $A$ is the
commutative
ring $k[x_1,x_2]$). After a linear transformation of
$X$, we may 
assume either that $a_{1212}=a_{1221}=0$ or that
$a_{1212}=0,
a_{1221}=1$.
\item
The 4 mixing relations can also be written as 
$$Yx_i=\begin{pmatrix} y_1\\y_2\end{pmatrix} x_i
=M_{i1}  \begin{pmatrix} x_1y_1\\x_1y_2\end{pmatrix}
+M_{i2} \begin{pmatrix} x_2y_1\\x_2y_2\end{pmatrix}
=M_{i1}  x_1Y +M_{i2} x_2Y$$
for $i=1,2$.
\item
If $Y$ is changed to $Y'=
BY$ where $B=(b_{ij})_{2\times 2}$ is an invertible
matrix,
then $M'$ is equal to $\begin{pmatrix}
B&0\\0&B\end{pmatrix} M \begin{pmatrix}
B^{-1}&0\\0&B^{-1}\end{pmatrix}$.
\end{enumerate}
\end{lemma} 

Since $\sigma$ is an algebra homomorphism, we have,
for $i,j,f,g$,
\begin{align}
\sigma_{ij}(x_f x_g)&=\sum_{p=1}^2
\sigma_{ip}(x_f)\sigma_{pj}(x_g)
\notag\\
&=\sum_{p,s,t=1}^2 (a_{ipfs} x_s)(a_{pjgt}x_t)
=\sum_{p,s,t=1}^2 (a_{ipfs}a_{pjgt})x_sx_t
\notag\\
&=(\sum_{p=1}^2 a_{ipf1}a_{pjg1}) x_1^2 +(\sum_{p}^2
a_{ipf1}a_{pjg2})x_1x_2
\notag\\
&\quad 
+(\sum_{p}^2 a_{ipf2}a_{pjg1})x_2x_1+(\sum_{p}^2
a_{ipf2}a_{pjg2})x_2^2
\notag
\end{align}
Using the relation $x_2x_1=q_{11}x_1^2+q_{12}x_1x_2$ in $A$, we
obtain
\begin{align}
\label{E3.1.1}
\sigma_{ij}(x_f x_g)&=
[(a_{i1f1}a_{1jg1}+a_{i2f1}a_{2jg1}) +
q_{11}(a_{i1f2}a_{1jg1}+ a_{i2f2}a_{2jg1})] x_1^2
\tag{E3.1.1} \\
&\quad +[(a_{i1f1}a_{1jg2}+a_{i2f1}a_{2jg2})+
q_{12}(a_{i1f2}a_{1jg1}+a_{i2f2}a_{2jg1})]
x_1x_2
\notag \\
&\quad +(a_{i1f2}a_{1jg2}+a_{i2f2}a_{2jg2})x_2^2.
\notag
\end{align}
Since $x_2x_1=q_{11}x_1^2+q_{12}x_1x_2$ and since
each $\sigma_{ij}$ is a $k$-linear map,
\begin{equation}
\label{E3.1.2}
\sigma_{ij}(x_2x_1)=q_{11}\sigma_{ij}(x_1x_1)+q_{12}\sigma_{ij}(x_1x_2)
\tag{E3.1.2}
\end{equation}
for all $i,j=1,2$. Now by \eqref{E3.1.1} we can
express the left-hand
and the right-hand sides of \eqref{E3.1.2} as
polynomials of $x_1$ and
$x_2$. By comparing coefficients of $x_1^2$, $x_1x_2$
and $x_2^2$ 
respectively, we obtain the following identities. The
coefficients of
$x_1^2$ of \eqref{E3.1.2} give rise to a constraint between 
coefficients 
\begin{align}
\label{C1ij}
(a_{i121}&a_{1j11}+a_{i221}a_{2j11}) +
q_{11}(a_{i122}a_{1j11}+ a_{i222}a_{2j11})
\tag{C1ij}\\
&=q_{11}[(a_{i111}a_{1j11}+a_{i211}a_{2j11}) +
q_{11}(a_{i112}a_{1j11}+ a_{i212}a_{2j11})]
\notag\\
&\quad +q_{12}[(a_{i111}a_{1j21}+a_{i211}a_{2j21}) +
q_{11}(a_{i112}a_{1j21}+ a_{i212}a_{2j21})].
\notag
\end{align}
The letter $C$ in \eqref{C1ij} stands for the {\it C}onstraints 
on the {\it C}oefficients. The coefficients of $x_1x_2$ of 
\eqref{E3.1.2} give rise to 
\begin{align}
\label{C2ij}
(a_{i121}&a_{1j12}+a_{i221}a_{2j12})+
q_{12}(a_{i122}a_{1j11}+a_{i222}a_{2j11})
\tag{C2ij}\\
&=q_{11}[(a_{i111}a_{1j12}+a_{i211}a_{2j12})+
q_{12}(a_{i112}a_{1j11}+a_{i212}a_{2j11})]
\notag\\
&\quad +q_{12}[(a_{i111}a_{1j22}+a_{i211}a_{2j22})+
q_{12}(a_{i112}a_{1j21}+a_{i212}a_{2j21})].
\notag
\end{align}
The coefficients of $x_2^2$ of \eqref{E3.1.2} gave
rise to 
\begin{align}
\label{C3ij}
(a_{i122}&a_{1j12}+a_{i222}a_{2j12})
\tag{C3ij}\\
&=q_{11}(a_{i112}a_{1j12}+a_{i212}a_{2j12})
+q_{12}(a_{i112}a_{1j22}+a_{i212}a_{2j22}).\qquad 
\notag
\end{align}

Next we apply (R3.1), (R3.2) and (R3.3) to the elements
$r=x_1$ and $x_2$, and obtain more relations between $a_{ijkl}$.
For $i,f,g,s,t=1,2$, 
\begin{align}
\label{E3.1.3}
\sigma_{fg}(\sigma_{st}(x_i))&=\sigma_{fg}(\sum_{w=1}^2
a_{stiw}x_w)
=\sum_{w=1}^2a_{stiw}\sigma_{fg}(x_w)
\tag{E3.1.3}\\
&=\sum_{w=1}^2a_{stiw} \sum_{j=1}^2 a_{fgwj}x_j=
\sum_{j=1}^2(a_{sti1}a_{fg1j}+a_{sti2}a_{fg2j})x_j
\notag\\
&=(a_{sti1}a_{fg11}+a_{sti2}a_{fg21})x_1+
(a_{sti1}a_{fg12}+a_{sti2}a_{fg22})x_2.
\notag
\end{align}
Recall that $P=(p_{12},p_{11})$; and we will set
$P=(1,1)$ or $(p,0)$
when we do the computation later. Using \eqref{E3.1.3}
relations  (R3.1)-(R3.3) (when applied to $x_i$) are 
equivalent to the following constraints on coefficients:
\begin{align}
\label{C4ij}
(a_{11i1}a_{211j}&+a_{11i2}a_{212j})+p_{11}
(a_{11i1}a_{221j}+a_{11i2}a_{222j})
\tag{C4ij}\\
&=p_{11}(a_{11i1}a_{111j}+a_{11i2}a_{112j})
+p_{11}^2(a_{11i1}a_{121j}+a_{11i2}a_{122j})
\notag\\
&\quad +p_{12}(a_{21i1}a_{111j}+a_{21i2}a_{112j})
+p_{11}p_{12}(a_{21i1}a_{121j}+a_{21i2}a_{122j}).
\notag
\end{align}
\begin{align}
\label{C5ij}
(a_{12i1}a_{211j}&+a_{12i2}a_{212j})+p_{12}
(a_{11i1}a_{221j}+a_{11i2}a_{222j})
\tag{C5ij}\\
&=p_{11}(a_{12i1}a_{111j}+a_{12i2}a_{112j})
+p_{11}p_{12}(a_{11i1}a_{121j}+a_{11i2}a_{122j})
\notag\\
&\quad +p_{12}(a_{22i1}a_{111j}+a_{22i2}a_{112j})
+p_{12}^2(a_{21i1}a_{121j}+a_{21i2}a_{122j}).
\notag
\end{align}
\begin{align}
\label{C6ij}
(a_{12i1}a_{221j}&+a_{12i2}a_{222j})
\tag{C6ij}\\
&=p_{11}(a_{12i1}a_{121j}+a_{12i2}a_{122j})
+p_{12}(a_{22i1}a_{121j}+a_{22i2}a_{122j}).
\notag
\end{align}

Note that there is a symmetry between the first three 
$C$-constraints (\eqref{C1ij}, \eqref{C2ij}, \eqref{C3ij}) 
and the last three $C$-constraints (\eqref{C4ij}, 
\eqref{C5ij}, \eqref{C6ij}). By \cite[Proposition 1.11]{ZZ},
if the coefficients $\{a_{ijfg}\}$ satisfy the six $C$-constraints, 
then six quadratic relations \eqref{MR11}-\eqref{MR22}, \eqref{NRx}
and \eqref{NRy} define a double extension 
$(k_{Q}[x_1,x_2])_P[y_1,y_2;\sigma]$. 
By \cite[Lemma 1.9]{ZZ}, $\sigma$ must be invertible, or equivalently,
the matrix $\Sigma$ must be invertible. So there is another constraint
in the coefficients: $\det \Sigma\neq 0$. By {\bf System C} we 
mean the system of equations \eqref{C1ij}, 
\eqref{C2ij}, \eqref{C3ij}, \eqref{C4ij},\eqref{C5ij}
and \eqref{C6ij} together with $\det \Sigma\neq 0$. We 
first fix $P=(p_{12},p_{11})$ and $Q=(q_{12},q_{11})$.
A {\bf solution to System C} or {\bf C-solution} is a
matrix $\Sigma$ with entries $a_{ijst}$ satisfying System C. 

Next we introduce some equivalence relations between
C-solutions.

\begin{lemma}
\label{xxlem3.2} Let $\Sigma$ be a C-solution and let
$B=(k_{Q}[x_1,x_2])_P[y_1,y_2;\sigma]$ where $\sigma$
is determined by $\Sigma$. Let $0\neq h\in k$ 
\begin{enumerate}
\item
$B$ is a ${\mathbb Z}^2$-graded algebra.
Let $\gamma: x_i\to x_i, y_i\to hy_i$. 
Then $\gamma$ extends to a graded automorphism of $B$.
\item
$h\Sigma$ is a C-solution. Let $\sigma'$ be the 
algebra automorphism determined by $h\Sigma$.
Then $B':=(k_{Q}[x_1,x_2])_P[y_1,y_2:\sigma']$
is a graded twist of $B$ in the sense of \cite{Zh1}.
\end{enumerate}
\end{lemma}

\begin{proof} (a) The assertion is clear. 

(b) Since relations (Csij) are homogeneous and
since $\det (h\Sigma)=h^4\det \Sigma$, 
$\Sigma$ is a C-solution if and only if $h\Sigma$
is. The second assertion can be verified by working
on the relations of the twist $B^\gamma$. 
\end{proof}

In general the algebra $B$ and its twist $B^\gamma$ are
not isomorphic as algebras. However these algebras have
many common properties since the category of 
graded $B$-modules is equivalent to the category of
graded $B^\gamma$-modules (see \cite{Zh1}). 

\begin{definition}
\label{xxdefn3.3}
We say $\Sigma$ and $\Sigma'$ are {\it twist equivalent} if 
$\Sigma'=h\Sigma$ for some $0\neq h\in k$.  In this case we 
say $\Sigma$ is a {\it twist} of $\Sigma$. It is easy
to see that twist equivalence is an equivalence relation. By
Lemma \ref{xxlem3.2}(b), we can replace $\Sigma$ by its
twists (without changing $Q$ and $P$) to obtain another 
double extension.

We say $(\Sigma, Q,P)$ and $(\Sigma',Q',P')$ are
{\it linear equivalent} if there is a graded algebra
isomorphism from $(k_{Q}[x_1,x_2])_P[y_1,y_2;\sigma]$
to $(k_{Q'}[x'_1,x'_2])_{P'}[y'_1,y'_2;\sigma']$
mapping $kx_1+kx_2\to kx'_1+kx'_2$ and $ky_1+ky_2\to ky'_1+ky'_2$.
Using this isomorphism we can pull back $x'_i$ and
$y'_i$ to the algebra $(k_{Q}[x_1,x_2])_P[y_1,y_2;\sigma]$;
then we can assume that $\{x'_1,x'_2\}$ (respectively, 
$\{y'_1,y'_2\}$) is another basis of $\{x_1,x_2\}$ 
(respectively, $\{y_1,y_2\}$). In case $Q=Q'$ and
$P=P'$, then we just say that $\Sigma$ and $\Sigma'$ are 
{\it linear equivalent}. We say $(\Sigma, Q,P)$ and 
$(\Sigma',Q',P')$ are {\it equivalent} if $(\Sigma,
Q,P)$ and $(h\Sigma',Q',P')$ are linear equivalent for some
$0\neq h\in k$.
\end{definition}

The following lemma is clear.

\begin{lemma}
\label{xxlem3.4} 
\begin{enumerate}
\item
Twist equivalence is an equivalence relation.
\item
Linear equivalence is an equivalence relation.
\item
Equivalence between $(\Sigma, Q,P)$ and 
$(\Sigma',Q',P')$ defined above is an equivalence
relation.
\end{enumerate}
\end{lemma}

Since our goal is to classify $A_P[y_1,y_2;\sigma]$ up
to isomorphism (or even up to twist), we will classify 
$\Sigma$ up to (linear) equivalence. Here are the 
strategies before we move into complicated computations
in the next section:

\noindent
{\bf Strategies:}

Strategy 1. We will use mathematical software Maple as
much as possible to reduce the length of the computations. 
The process of solving the System C by Maple and the 
corresponding codes will be omitted since the codes are 
very simple. Not all solutions will be listed here since the list 
of the solutions to System C is still large. We need 
to do more reductions in next two steps to achieve our 
final solution. Even then we will see a large number of
solutions. 

Strategy 2. We will try not to analyze iterated Ore extensions. 
In many cases, when a C-solution gives rise to an iterated 
Ore extension, we will stop. To test when a $\Sigma$ gives
rise an an iterated Ore extension $A_P[y_1,y_2;\sigma,\delta,
\tau]$ we will mainly use Proposition \ref{xxprop3.5} below.
Most of such solutions will not be listed; but a few examples
will be given in subsection \ref{xxsubsec4.1}.

Strategy 3. Further reductions will be done by using equivalence 
relations between $(\Sigma,Q,P)$ and $(\Sigma',Q',P')$. 
There is no unique way of doing this since linear
equivalence is dependent on particular choices of $Q$ and $P$.
This is one of the reason we break the computation into four
cases in the next four subsections according to the form of $P$.

\begin{proposition}
\label{xxprop3.5} Let $B$ be a double extension 
$(k_{Q}[x_1,x_2])_P[y_1,y_2;\sigma,\delta,\tau]$.
\begin{enumerate}
\item
If $\Sigma_{12}=0$, then $B$ is an iterated Ore
extension.
\item
If $\Sigma_{21}=0$ and $p_{11}=0$, then $B$ is an
iterated Ore extension. 
\end{enumerate}
\end{proposition}

\begin{proof} (a) If $\Sigma_{12}=0$, then
$\sigma_{12}=0$.
Hence the first half of the relation
\eqref{R1} becomes
$$y_1 r=\sigma_{11}(r)y_1+\delta_1(r)$$
for all $r\in A:=k_Q[x_1,x_2]$. It is easy to check
that $\sigma_{11}$ is an automorphism of $A$ and
$\delta_1$ is a $\sigma$-derivation of $A$. Therefore 
the subalgebra generated by $x_1,x_2,y_1$ is an Ore 
extension of $A$. The second half of \eqref{R1} together 
with \eqref{R1} shows that $B$ is an Ore extension of
$A[y_1;\sigma_{11},\delta_1]$. The assertion follows.

(b) Since $p_{11}=0$, we switch $y_1$ and $y_2$ without 
changing the form of \eqref{R1} , $\Sigma_{21}=0$ becomes 
$\Sigma_{12}=0$. The assertion follows from (a). Note that 
if $p_{11}\neq 0$, then we can not switch $y_1$ and $y_2$ 
to keep the form of \eqref{R1}.
\end{proof}

The following proposition is a consequence of above.

\begin{proposition}
\label{xxprop3.6}
Let $B$ be a trimmed double extension 
$(k_{Q}[x_1,x_2])_P[y_1,y_2;\sigma]$ where $\sigma$
is determined by the matrix $\Sigma$.
\begin{enumerate}
\item
Considering $A':=k_{P}[y_1,y_2]$ as the subring and 
$\{x_1,x_2\}$ is the set of generators over $A'$,
 $B$ is a double extension of $(k_{P}[y_1,y_2])_{Q}[x_1,x_2;\alpha]$
where $\alpha$ is determined by the matrix $M^{-1}$.
\item
If $M_{12}=0$, then $B$ is an iterated Ore extension 
of $k_{P}[y_1,y_2]$.
\item
If $M_{21}=0$ and $q_{11}=0$, then $B$ is an iterated 
Ore extension of $k_{P}[y_1,y_2]$.
\end{enumerate}
\end{proposition}

\begin{proof} (a) We need to switch the roles played by $x_i$ and
$y_i$. Four mixing relations can be written as
\begin{equation}
\label{E3.6.1}
\begin{pmatrix}
y_1x_1\\y_2x_1\\y_1x_2\\y_2x_2
\end{pmatrix}=M\begin{pmatrix}
x_1y_1\\x_1y_2\\x_2y_1\\x_2y_2
\end{pmatrix}.
\tag{E3.6.1}
\end{equation}
This implies that
$$\begin{pmatrix}
x_1y_1\\x_1y_2\\x_2y_1\\x_2y_2
\end{pmatrix}=M^{-1}\begin{pmatrix}
y_1x_1\\y_2x_1\\y_1x_2\\y_2x_2
\end{pmatrix}.$$
Therefore $B$ is the double extension
$(k_{P}[y_1,y_2])_{Q}[x_1,x_2;\alpha]$
where $\alpha$ is determined by the matrix
$M^{-1}$.

(b,c) By part (a), the matrix $M^{-1}$ plays
the role of $\Sigma$-matrix (if we
switch $x_i$ with $y_i$). Since $M_{12}=0$
(respectively, $M_{21}=0$)
if and only if $(M^{-1})_{12}=0$ 
(respectively, $(M^{-1})_{21}=0$), the assertions
follows from Proposition \ref{xxprop3.5}(a,b). 
\end{proof}

The matrix $M$ also appears in a slightly different
setting, see the next proposition. For any $P=(p_{12},p_{11})$. Let
$P^{\circ}$ denote the set $(p_{12}^{-1},-p_{12}^{-1}p_{11})$.

\begin{proposition}
\label{xxprop3.7}
The opposite ring of $B:=(k_{Q}[x_1,x_2])_P[y_1,y_2;\sigma]$ 
is a double extension $(k_{P^{\circ}}[y_1,y_2])_{Q^{\circ}}
[x_1,x_2;\xi]$ where $\xi$ is determined by the matrix $M$.
\end{proposition}

\begin{proof} Let $\star$ be the multiplication of 
the opposite ring $B^{op}$. The relation 
$$y_2y_1=p_{12}y_1y_2+p_{11}y_1^2$$
in $B$ implies the relation
$$y_2\star y_1=p_{12}^{-1}y_1\star y_2+(-p_{12}^{-1}p_{11})y_1^2$$
in $B^{op}$. The same is true for the relation between $x_1$ and $x_2$.
The relations \ref{E3.6.1} in $B$ implies the following
relations in $B^{op}$
$$\begin{pmatrix}
x_1\star y_1\\x_1\star y_2\\x_2\star y_1\\x_2\star y_2
\end{pmatrix}=M\begin{pmatrix}
y_1\star x_1\\y_2\star x_1\\y_1\star x_2\\y_2\star x_2
\end{pmatrix}.$$
Recall that $x_i$ and $y_i$ are switched in the double extension 
$(k_{P^{\circ}}[y_1,y_2])_{Q^{\circ}} [x_1,x_2;\xi]$.
So the matrix $M$ plays the role of $\Sigma$-matrix for 
the homomorphism $\xi$.
\end{proof}

\section{A classification of $\{\Sigma,P,Q\}$}
\label{xxsec4}

Now we start our classification. 

\subsection{Case one: $P=(1,1)$}
\label{xxsubsec4.1}
In this subsection we classify $\Sigma$ when
$P=(1,1)$. 
We consider the following subcases.

\bigskip

\noindent
{\bf Subcase 4.1.1:} $Q=(1,1)$. The System C is solved
by Maple 
to give two solutions in this case.

Solution one: $\Sigma=\begin{pmatrix}
f&   0&   0&  0\\
g&   f&   0&  0\\
h&   0&   f&  0\\
m&   h&   g&  f\end{pmatrix}$ where $f\neq 0$. 

Solution two:
$\Sigma=\begin{pmatrix}
f&   0&   0&  0\\
g&   f&   0&  0\\
-f^2/g&   0&   f&  0\\
h& -f(f-m+g)/g&   m&  f\end{pmatrix}$ where $fg\neq
0$.

In both solutions we have $\Sigma_{12}=0$. By
Proposition \ref{xxprop3.5}(a), these $\Sigma$ will 
produce iterated Ore extensions. By Strategy 2, we will 
not study these algebras further in this paper.

\bigskip

\noindent
{\bf Subcase 4.1.2:} $Q=(q,0)$ where $q\neq 0, \pm 1$.
The System C is solved by Maple to give one solution, 
which is
$\Sigma=\begin{pmatrix}
f&   0&   0&  0\\
0&   g&   0&  0\\
h&   0&   f&  0\\
0&   m&   0&  g
\end{pmatrix}$ where $fg\neq 0$. Again since
$\Sigma_{12}=0$, 
by Proposition \ref{xxprop3.5}, we will only obtain an
iterated Ore extension.

\bigskip

\noindent
{\bf Subcase 4.1.3:} $Q=(-1,0)$. There are two C-solutions.
One is the same as the solution in Subcase 4.1.2; so it 
gives rise to an iterated Ore extension. The other is 
$\Sigma=\begin{pmatrix}
0&   f&   0&  0\\
g&   0&   0&  0\\
0&   fh/g&   0&  f\\
h&   0&   g&  0
\end{pmatrix}$ where $fg\neq 0$. Again in this case we
only obtain an iterated Ore extension. Up to this point
we only used Strategies 1 and 2. 

\bigskip
\noindent
{\bf Subcase 4.1.4:} $Q=(1,0)$. The System C is solved
by Maple to give 15 solutions, 13 of which has the property
$\Sigma_{12}=0$. To save the space we will not list these 
solutions. Next we use Strategy 3. 

Since $Q=(1,0)$, we can make linear transformation of
$\{x_1,x_2\}$.
By Lemma \ref{xxlem3.1}(c) we may further assume that 
either $a_{1212}=a_{1221}=0$ or
$a_{1212}=0,a_{1221}=1$. 

If $a_{1212}=0$ and $a_{1221}=0$, the System C is
solved by Maple to give 15 solutions, all of which have 
the property that 
$\Sigma_{12}=0$. So we only consider the case when
$a_{1212}=0$ 
and $a_{1221}=1$. The the System C is solved to give
a single solution:
$\Sigma=\begin{pmatrix}
f&   0&   0&  0\\
g&   f&   1&  0\\
0&   0&   f&  0\\
m&   -2f&   -g-1&  f
\end{pmatrix}$ or equivalently 
$\Sigma=\begin{pmatrix}
f&   0&   0&  0\\
g&   f&   f&  0\\
0&   0&   f&  0\\
m&   -2f&   -g-f&  f
\end{pmatrix}$. Up to a twist equivalence, we may
assume 
that $\Sigma=\begin{pmatrix}
1&   0&   0&  0\\
g&   1&   1&  0\\
0&   0&   1&  0\\
m&   -2&   -g-1&  1
\end{pmatrix}$.
Now we will make linear transformation of
$Y=(y_1,y_2)^T$.
It is a bit easier to see this by using matrx $M$. The
last
$\Sigma$ is equivalent to 
$M=\begin{pmatrix}
1&   0&   0&  0\\
0&   1&   0&  0\\
g&   1&   1&  0\\
m&   -g-1&   -2&  1
\end{pmatrix}$.
Since $P=(1,1)$, we can change $Y$ to $Y'=BY$ where 
$B=\begin{pmatrix} 1&0\\g&1\end{pmatrix}$. By
doing so the structure of the relations will not
change, but the matrix $M$ becomes
$\begin{pmatrix}
1&   0&   0&  0\\
0&   1&   0&  0\\
0&   1&   1&  0\\
m'&   -1&   -2&  1
\end{pmatrix}$ where $m'=m+g+g^2$. Let $g$ denote the
new $m'$. 
We have $\Sigma=\begin{pmatrix}
1&   0&   0&  0\\
0&   1&   1&  0\\
0&   0&   1&  0\\
g&   -2&   -1&  1
\end{pmatrix}$. Now we make another linear
transformation
$X'=BX$ where $B=\begin{pmatrix}
1&0\\g/2&1\end{pmatrix}$,
then we have $\Sigma=\begin{pmatrix}
1&   0&   0&  0\\
0&   1&   1&  0\\
0&   0&   1&  0\\
0&   -2&   -1&  1
\end{pmatrix}$. This is the only possible $\Sigma$ up
to (linear and twist) equivalence. Therefore up to 
linear equivalence, we have the first interesting case

\noindent
Algebra $\AAA$: $\Sigma=h\begin{pmatrix}
1&   0&   0&  0\\
0&   1&   1&  0\\
0&   0&   1&  0\\
0&   -2&   -1&  1
\end{pmatrix}$ and 
$M=h\begin{pmatrix}
1&   0&   0&  0\\
0&   1&   0&  0\\
0&   1&   1&  0\\
0&   -1&   -2&  1
\end{pmatrix}$ 
where $h\neq 0$; and $P=(1,1)$, $Q=(1,0)$. 
In the rest of the section let $h$ be a 
nonzero scalar in $k$. We can easily write
down the relations of the algebra $\AAA$ from 
$\{\Sigma,P,Q\}$. The matrix $\Sigma$ gives 
us the four mixing relations between $x_i$ 
and $y_i$. The $Q$ tells us
the relation between $x_1$ and $x_2$ and 
the $P$ tells us the relation between $y_1$
and $y_2$. Here are the six quadratic relations
of the algebra $\AAA$:
$$\begin{aligned}
x_2x_1&=x_1x_2\\
y_2y_1&=y_1y_2+y_1^2\\
y_1x_1&=x_1y_1\\
y_1x_2&=x_2y_1+x_1y_2\\
y_2x_1&=x_1y_2\\
y_2x_2&=-2x_2y_1-x_1y_2+x_2y_2.
\end{aligned}
$$
All algebras in this section are generated 
by $x_1,x_2,y_1$ and $y_2$. To save space, we will not
write down explicitly the relations of other
algebras except for the algebra $\ZZ$ at the end of
this section.

By Proposition \ref{xxprop3.6}(a) any double extension 
$(k_{Q}[x_1,x_2])_{P}[y_1,y_2;\sigma]$ is
isomorphic to $(k_{P}[y_1,y_2])_{Q}[x_1,x_2;\alpha]$ 
where $\alpha$ is determined by the matrix $M^{-1}$. 
In the case of the algebra $\AAA$, we have $M_{12}=0$.
By Proposition \ref{xxprop3.6}(b) $\AAA$ is 
an iterated Ore extension of $k_{P}[y_1,y_2]$. 
However, there are possible $\delta,\tau$ such that 
$(k_{Q}[x_1,x_2])_{P}[y_1,y_2;\sigma,\delta,\tau]$ is not 
an iterated Ore extension of any regular algebra of dimension 2. 
This is the reason the algebra $\AAA$ is not deleted from
our 26 families.

The point-scheme of the algebra $\AAA$ can be computed. We see 
that the dimension of the point-scheme is 1 in this case and
details are omitted. Recall from \cite[Section 3]{ZZ} that the 
determinant of $\sigma$ is defined to be
$$\det \sigma=-p_{11}\sigma_{12}\sigma_{11}+\sigma_{22}
\sigma_{11}-p_{12}\sigma_{12}\sigma_{21}$$
which is an algebra automorphism of $k_{Q}[x_1,x_2]$.
For the algebra $\AAA$ we have
$$
\det \sigma
\begin{pmatrix}
x_1\\x_2\end{pmatrix}
=h^2\begin{pmatrix}
1 & 0\\0 & 1\end{pmatrix}
\begin{pmatrix}
x_1\\x_2\end{pmatrix}$$

This is the end of classification of $\{\Sigma,Q\}$ when
$P=(1,1)$. 

\bigskip

\subsection{Case two: $P=(p,0)$ where $p\neq \pm 1$.} 
We consider four subcases as in Case one. Some arguments are
similar to the ones given in Case (section \ref{xxsubsec4.1}) 
one above, so most of the details will be omitted. 

\bigskip
\noindent
{\bf Subcase 4.2.1:} $Q=(1,1)$. There is only one C-solution in which 
$\Sigma_{12}=0$. By Proposition \ref{xxprop3.5}(a) we only get an 
iterated Ore extension.

\bigskip
\noindent
{\bf Subcase 4.2.2:} $Q=(q,0)$ where $q\neq \pm 1$. The proof 
of the following lemma is based on tedious computation and 
can be verified by Maple very quickly. 

\begin{lemma}
\label{xxlem4.1}
Suppose $P=(p,0)$ and $Q=(q,0)$ where $p\neq \pm 1$
and $q\neq \pm1$. Suppose $\Sigma$ is a C-solution 
(in particular $\det \Sigma\neq 0$).
\begin{enumerate}
\item
If $p\neq \pm i, \xi_3, \xi_3^2$ where $i$ is the primitive 
4th root of $1$ and $\xi_3$ is the primitive
3rd root of 1, then either $\Sigma_{12}=0$ or
$\Sigma_{21}=0$.
\item
Suppose $\Sigma_{12}\neq 0$ and $\Sigma_{21}\neq 0$.
\begin{enumerate}
\item[(i)]
If $p^2=-1$ (or $p=\pm i$), then either $q=p$
or $q=p^{-1}$.
\item[(ii)]
If $p=\xi_3$ or $\xi_3^2$, then either 
$q=p$ or $q=p^{-1}$. 
\item[(iii)]
After exchanging $x_1$ and $x_2$, we may assume that
$q=p$.
\end{enumerate}
\end{enumerate}
\end{lemma}

The next lemma follows from Lemma \ref{xxlem3.1}.

\begin{lemma} 
\label{xxlem4.2}
We fix $P=(p,0)$ and $Q=(q,0)$. Let $\Sigma$ be a 
C-solution with $\Sigma_{ij}=(a_{ijst})_{2\times 2}$. 
\begin{enumerate}
\item
If the basis $\{x_1,x_2\}$ is changed to
$\{x_1,ax_2\}$,
then the entry $a_{ijst}$ of $\Sigma$ is changed to 
$a^{(s-t)/2}a_{ijst}$.
\item
If the basis $\{y_1,y_2\}$ is changed to
$\{y_1,by_2\}$,
then the entry $a_{ijst}$ of $\Sigma$ is changed to 
$b^{(i-j)/2}a_{ijst}$.
\item
If $a_{1211}\neq 0$, after a linear transformation of 
$\{y_1,y_2\}$, we may assume $a_{1211}=1$.
\item
If $a_{1221}\neq 0$, after a linear transformation of 
$\{x_1,x_2\}$ (or $\{y_1,y_2\}$), we may assume
$a_{1221}=1$.
\end{enumerate}
\end{lemma}

According to Lemma \ref{xxlem4.2} we may assume that 
$a_{1211}=0$, or $1$ and $a_{1221}=0$, or $1$. From now 
on we will only consider those solutions with $\Sigma_{12}\neq 0$. 
(If $\Sigma_{12}=0$, then use Proposition \ref{xxprop3.5}(a).)
We may further assume that the first column of $\Sigma_{12}$
is nonzero. If the second column is nonzero, then by
switch $x_1$ with $x_2$ we obtain that the first column of 
$\Sigma_{12}$ is nonzero. Hence there are three cases
to consider (up to a linear equivalence):

Case (i): $a_{1211}=1$, $a_{1221}=0$. The Maple gives
no C-solution.

Case (ii): $a_{1211}=0$, $a_{1221}=1$. The Maple gives
two C-solutions. One of which has $\Sigma_{21}=0$, so 
we omit this one by Proposition \ref{xxprop3.5}(b). 
The other is 
$\Sigma=\begin{pmatrix}
0&   0&  0&   f\\
0&   0&  1&   0\\
0&  -fg& 0&   0\\
g&   0&  0&   0
\end{pmatrix}$ with $fg\neq 0$, and $p=q$ and $p^2=-1$. Using
Lemma \ref{xxlem4.2} above $\Sigma$ is linearly equivalent to

Algebra $\BB$: $\Sigma=h\begin{pmatrix}
0&   0&  0&   1\\
0&   0&  1&   0\\
0&  -1& 0&   0\\
1&   0&  0&   0
\end{pmatrix}$ and $M=h\begin{pmatrix}
0&   0&  0&   1\\
0&   0&  -1&   0\\
0&  1& 0&   0\\
1&   0&  0&   0
\end{pmatrix}$;
and $P=(p,0)=Q$ and $p^2=-1$. 

The determinant of $\sigma$ is
$$\det \sigma
\begin{pmatrix}
x_1\\x_2\end{pmatrix}
=ph^2\begin{pmatrix}
1 & 0\\0 & -1\end{pmatrix}
\begin{pmatrix}
x_1\\x_2\end{pmatrix}.$$

Recall that $P^{\circ}$ denotes the set 
$(p_{12}^{-1},-p_{12}^{-1}p_{11})$ where 
$P=(p_{12},p_{11})$.

\begin{definition}
\label{xxdefn4.3}
Let $(k_{Q}[x_1,x_2])_{P}[y_1,y_2;\sigma]$ and 
$(k_{Q'}[x_1,x_2])_{P'}[y_1,y_2;\sigma']$ be two
double extensions.
\begin{enumerate}
\item
Two double extensions are called {\it $\Sigma$-$M$-dual} 
if the matrix $\{M,Q^{\circ},P^{\circ}\}$ is equivalent 
to $\{\Sigma',P',Q'\}$ in the sense of Definition 
\ref{xxdefn3.3}.
\item
A double extension is called {\it $\Sigma$-$M$-selfdual}
if $\{M,Q^{\circ},P^{\circ}\}$ is equivalent to 
$\{\Sigma,P,Q\}$ in the sense of Definition \ref{xxdefn3.3}.
\end{enumerate}
\end{definition}

By Proposition \ref{xxprop3.7}, if 
$(k_{Q}[x_1,x_2])_{P}[y_1,y_2;\sigma]$ and 
$(k_{Q'}[x_1,x_2])_{P'}[y_1,y_2;\sigma']$ are $\Sigma$-$M$-dual,
then  
$$(k_{Q}[x_1,x_2])_{P}[y_1,y_2;\sigma]^{\gamma}
\cong ((k_{Q'}[x_1,x_2])_{P'}[y_1,y_2;\sigma'])^{op}$$ 
for some automorphism twist $\gamma$. In particular, if 
$(k_{Q}[x_1,x_2])_{P}[y_1,y_2;\sigma]$ is $\Sigma$-$M$-selfdual
then 
$$(k_{Q}[x_1,x_2])_{P}[y_1,y_2;\sigma]^{\gamma}
\cong (k_{Q}[x_1,x_2])_{P}[y_1,y_2;\sigma]^{op}$$
for some $\gamma$.

It is easy to verify that the algebra $\BB$ is $\Sigma$-M-selfdual.
Also the algebra $\BB$ contains two cases with the same
matrix $\Sigma$, namely, $p=i$ and $p=-i$. 

Case (iii): $a_{1211}=1$, $a_{1221}=1$. There are four C-solutions
such that $\Sigma_{21}\neq 0$. All four solutions are linearly 
equivalent in the sense of Definition \ref{xxdefn3.3}. Here we use 
the linear equivalences of the form  $(\Sigma, (p,0),(q,0))\sim 
(\Sigma',(p',0),(q',0))$ where $p'$ is $p$ or $p^{-1}$
and $q'$ is $q$ or $q^{-1}$. So up to linear equivalences, 
we only have one C-solution:

Algebra $\CC$: 
$\Sigma=h \begin{pmatrix}
-1&   p^2&  1&   -p\\
-p&   1&  1&   -p\\
-p&  -2p^2& p&   -p\\
-p&   p^2&  1&   -1
\end{pmatrix}$ and $M=h\begin{pmatrix}
-1&   1&  p^2&   -p\\
-p&   p&  -2p^2&   -p\\
-p&  1& 1&   -p\\
-p&   1&  p^2&   -1
\end{pmatrix}$; and $P=(p,0)=Q$ and $p^2+p+1=0$. 

The determinant of $\sigma$ is 
$$\det \sigma
\begin{pmatrix}
x_1\\x_2\end{pmatrix}
=-3h^2\begin{pmatrix}
p & 0\\0 & 1\end{pmatrix}
\begin{pmatrix}
x_1\\x_2\end{pmatrix}.$$

It is easy to check that the algebra $\CC$ is 
$\Sigma$-M-selfdual. Of course the equation
$p^2+p+1=0$ has two solutions, and 
we may think the algebra $\CC$ contains two different 
cases. This is the end of {\bf Subcase 4.2.2}.

\bigskip
\noindent
{\bf Subcase 4.2.3:} $Q=(-1,0)$. Similar to the argument
given in {\bf Subcase 4.2.2} we need consider the following
three cases:
$(a_{1211},a_{1221})=(1,0)$ or
$(a_{1211},a_{1221})=(0,1)$ or 
$(a_{1211},a_{1221})=(1,1)$. 

Case (i) $a_{1211}=1$, $a_{1221}=0$. C-solutions have 
$\Sigma_{21}=0$. So we stop here.

Case (ii): $a_{1211}=0$, $a_{1221}=1$. Only one
C-solution has the property $\Sigma_{21}\neq 0$.

Algebra $\DD$: 
$\Sigma=h \begin{pmatrix}
-p&   0&  0&   0\\
0&   -p^2&  1&   0\\
0&  0& p&   0\\
1&   0&  0&   1
\end{pmatrix}$ and 
$M= h \begin{pmatrix}
-p&   0&  0&   0\\
0&   p&  0&   0\\
0&  1& -p^2&   0\\
1&   0&  0&   1
\end{pmatrix}$;
$P=(p,0)$ where $p$ is a general parameter which 
could be $\pm 1$ and $Q=(-1,0)$.

The determinant of $\sigma$ is 
$$\det \sigma
\begin{pmatrix}
x_1\\x_2\end{pmatrix}
=-p^2h^2\begin{pmatrix}
1 & 0\\0 & 1\end{pmatrix}
\begin{pmatrix}
x_1\\x_2\end{pmatrix}.$$

Case (iii): $a_{1211}=1$, $a_{1221}=1$.
There are two C-solutions such that $\Sigma_{21}\neq 0$.

Algebra $\EE$: 
$\Sigma=h \begin{pmatrix}
0&   0&  1&   1\\
0&   0&  1&   -1\\
-1&  1& 0&   0\\
1&   1&  0&   0
\end{pmatrix}$ and $M=h\begin{pmatrix}
0&   1&  0&   1\\
-1&   0&  1&   0\\
0&  1& 0&   -1\\
1&   0&  1&   0
\end{pmatrix}$; 
$P=(p,0)$ where $p^2=-1$ and $Q=(-1,0)$.

The determinant of $\sigma$ is 
$$\det \sigma
\begin{pmatrix}
x_1\\x_2\end{pmatrix}
=2ph^2\begin{pmatrix}
0 & 1\\-1 & 0\end{pmatrix}
\begin{pmatrix}
x_1\\x_2\end{pmatrix}.$$

Algebra $\FF$: 
$\Sigma=h \begin{pmatrix}
-1&   -p&  1&   -1\\
-p&   1&  1&   1\\
-p&  p& p&   1\\
-p&   -p&  1&   -p
\end{pmatrix}$ and $M=h\begin{pmatrix}
-1&   1&  -p&   -1\\
-p&   p&  p&   1\\
-p&  1& 1&   1\\
-p&   1&  -p&   -p
\end{pmatrix}$; 
$P=(p,0)$ where $p^2=-1$ and $Q=(-1,0)$.

The determinant of $\sigma$ is 
$$\det \sigma
\begin{pmatrix}
x_1\\x_2\end{pmatrix}
=-2ph^2\begin{pmatrix}
1 & 0\\0 & 1\end{pmatrix}
\begin{pmatrix}
x_1\\x_2\end{pmatrix}.$$

\bigskip
\noindent
{\bf Subcase 4.2.4:} $Q=(1,0)$. We can make linear
transformations of $X$ such that $\Sigma_{12}$ is
one of the standard forms:
$\Sigma_{12}=\begin{pmatrix}
a_1&0\\0&a_1\end{pmatrix}$ or
$\Sigma_{12}=\begin{pmatrix}
a&0\\1&a\end{pmatrix}$. In particular we may assume
either $(a_{1212},a_{1221})=(0,0)$ or
$(a_{1212},a_{1221})=(0,1)$.

Case (i): $(a_{1212},a_{1221})=(0,0)$. There is no 
C-solution with $\Sigma_{21}\neq 0$.

Case (ii): $(a_{1212},a_{1221})=(0,1)$. There is 
only one C-solution such that $\Sigma_{21}\neq 0$.

Algebra $\GG$: 
$\Sigma=h \begin{pmatrix}
p&    0&    0&   0\\
 p&  p^2&   1&   0\\
0 &    0&    p&  0\\
f&   0&  -1& 1
\end{pmatrix}$ and $M=h\begin{pmatrix}
p&    0&    0&   0\\
 0&  p&   0&   0\\
p &    1&    p^2&  0\\
f&   -1&  0& 1
\end{pmatrix}$ with $f\neq 0$;
$P=(p,0)$ where $p$ is general and $Q=(1,0)$.

The determinant of $\sigma$ is 
$$\det \sigma
\begin{pmatrix}
x_1\\x_2\end{pmatrix}
=p^2h^2\begin{pmatrix}
1 & 0\\0 & 1\end{pmatrix}
\begin{pmatrix}
x_1\\x_2\end{pmatrix}.$$

This is the end of Subsection 4.2 where 
$P=(p,0)$ and $p\neq 0,\pm 1$.

\subsection{Case three: $P=(-1,0)$.} 
\label{xxsubsec4.3}
Following steps before we consider four subcases.

\bigskip
\noindent
{\bf Subcase 4.3.1:} $Q=(1,1)$. Up to linear
equivalence, the System C has one solution 
with $\Sigma_{12}\neq 0$:

Algebra $\HH$:  $\Sigma=h \begin{pmatrix}
0&   0&   1&   0\\
0&   0&   f&   1\\
1&   0&   0&   0\\
f&   1&   0&   0
\end{pmatrix}$ and $M=h \begin{pmatrix}
0&   1&   0&   0\\
1&   0&   0&   0\\
0&   f&   0&   1\\
f&   0&   1&   0
\end{pmatrix}$;
$P=(-1,0)$ and $Q=(1,1)$.

The determinant of $\sigma$ is 
$$\det \sigma
\begin{pmatrix}
x_1\\x_2\end{pmatrix}
=h^2\begin{pmatrix}
1 & 0\\2f & 1\end{pmatrix}
\begin{pmatrix}
x_1\\x_2\end{pmatrix}.$$

\bigskip
\noindent
{\bf Subcase 4.3.2:} $Q=(q,0)$ where $q\neq \pm 1$.
The system C has four solutions up
to linear transformation with $\Sigma_{12}\neq 0$
and $\Sigma_{21}\neq 0$:

Algebra $\II$: 
$\Sigma=h\begin{pmatrix}
-q&    -q &  1&   -q\\
1&    1&   1&   -q\\
1&    q&   q&   -q\\
-1&   -q&   1&   -1  
\end{pmatrix}$ and $M=h\begin{pmatrix}
-q&    1 &  -q&   -q\\
1&    q&   q&   -q\\
1&    1&   1&   -q\\
-1&   1&   -q&   -1  
\end{pmatrix}$; $P=(-1,0)$ and $Q=(q,0)$ where
$q^2=-1$.

Algebra $\II$ is $\Sigma$-M-dual to algebra $\FF$.
The determinant of $\sigma$ is 
$$\det \sigma
\begin{pmatrix}
x_1\\x_2\end{pmatrix}
=2h^2\begin{pmatrix}
1 & 0\\0 & -1\end{pmatrix}
\begin{pmatrix}
x_1\\x_2\end{pmatrix}.$$

Algebra $\JJ$: 
$\Sigma=h\begin{pmatrix}
0&   1&   0&   1\\
-1&   0&   1&   0\\
0&   1&   0&   -1\\
1&   0&   1&   0
\end{pmatrix}$ and $M=\begin{pmatrix}
0&   0&   1&   1\\
0&   0&   1&   -1\\
-1&   1&   0&   0\\
1&   1&   0&   0
\end{pmatrix}$;
$P=(-1,0)$ and $Q=(q,0)$ where $q^2=-1$.

Algebra $\JJ$ is $\Sigma$-M-dual to algebra 
$\EE$. The determinant of $\sigma$ is 
$$\det \sigma
\begin{pmatrix}
x_1\\x_2\end{pmatrix}
=2h^2\begin{pmatrix}
1 & 0\\0 & 1\end{pmatrix}
\begin{pmatrix}
x_1\\x_2\end{pmatrix}.$$

Algebra $\KK$: 
$\Sigma=h\begin{pmatrix}
1&   0&   0&   0\\
0&   0&   0&   1\\
0&   0&   1&   0\\
0&   f&   0&   0
\end{pmatrix}$ and $M=h\begin{pmatrix}
1&   0&   0&   0\\
0&   1&   0&   0\\
0&   0&   0&   1\\
0&   0&   f&   0
\end{pmatrix}$ where $f\neq 0$; 
$P=(-1,0)$ and $Q=(q,0)$ where $q$ is a general parameter 
which could be $\pm 1$.

The determinant of $\sigma$ is 
$$\det \sigma
\begin{pmatrix}
x_1\\x_2\end{pmatrix}
=h^2\begin{pmatrix}
1 & 0\\0 & f\end{pmatrix}
\begin{pmatrix}
x_1\\x_2\end{pmatrix}.$$

Algebra $\LL$: $\Sigma=h\begin{pmatrix}
0&    0&   f&   0\\
0&    0&   0&   1\\
f&    0&   0&   0\\
0&   1&   0&   0
\end{pmatrix}$  and $M=h\begin{pmatrix}
0&    f&   0&   0\\
f&    0&   0&   0\\
0&    0&   0&   1\\
0&   0&   1&   0
\end{pmatrix}$ where $f\neq 0$; 
$P=(-1,0)$ and $Q=(q,0)$ where $q$ is a general parameter 
which could be $\pm 1$.

The determinant of $\sigma$ is 
$$\det \sigma
\begin{pmatrix}
x_1\\x_2\end{pmatrix}
=h^2\begin{pmatrix}
f^2 & 0\\0 & 1\end{pmatrix}
\begin{pmatrix}
x_1\\x_2\end{pmatrix}.$$

\bigskip

\noindent
{\bf Subcase 4.3.3:} $Q=(-1,0)$. There are nine
C-solutions up to linear transformation that have
$\Sigma_{12}\neq 0$ and $\Sigma_{21}\neq 0$.

Algebra $\MM$:
$\Sigma=h\begin{pmatrix}
0&   1&  1&  0\\
f&   0&  0&  -1\\
1&   0&  0&  -1\\
0&   -1&  -f&  0
\end{pmatrix}$ and $M=h\begin{pmatrix}
0&   1&  1&  0\\
1&   0&  0&  -1\\
f&   0&  0&  -1\\
0&   -f&  -1&  0
\end{pmatrix}$ where $f\neq 1$; $P=Q=(-1,0)$.

Algebra $\MM$ is $\Sigma$-M-selfdual.
The determinant of $\sigma$ is 
$$\det \sigma
\begin{pmatrix}
x_1\\x_2\end{pmatrix}
=(1-f)h^2\begin{pmatrix}
1 & 0\\0 & 1\end{pmatrix}
\begin{pmatrix}
x_1\\x_2\end{pmatrix}.$$

Algebra $\NN$:
$\Sigma=\begin{pmatrix}
0&   -g&  0&  f\\
g&   0&  f&  0\\
0&   f&  0&  -g\\
f&   0&  g&  0
\end{pmatrix}$ and $M=\begin{pmatrix}
0&   0&  -g&  f\\
0&   0&  f&  -g\\
g&   f&  0&  0\\
f&   g&  0&  0
\end{pmatrix}$
where $f^2\neq g^2$; $P=Q=(-1,0)$.

The determinant of $\sigma$ is 
$$\det \sigma
\begin{pmatrix}
x_1\\x_2\end{pmatrix}
=(f^2-g^2)\begin{pmatrix}
1 & 0\\0 & 1\end{pmatrix}
\begin{pmatrix}
x_1\\x_2\end{pmatrix}.$$

Algebra $\OO$:
$\Sigma=h\begin{pmatrix}
1&   0&  0&  f\\
0&   -1&  1&  0\\
0&   f&  -1&  0\\
1&   0&  0&  1
\end{pmatrix}$ and $M=h\begin{pmatrix}
1&   0&  0&  f\\
0&   -1&  f&  0\\
0&   1&  -1&  0\\
1&   0&  0&  1
\end{pmatrix}$
where $f\neq 1$; $P=Q=(-1,0)$. A special case 
is when $f=0$.
The algebra $\OO$ is $\Sigma$-M-selfdual.

The determinant of $\sigma$ is 
$$\det \sigma
\begin{pmatrix}
x_1\\x_2\end{pmatrix}
=(f-1)h^2\begin{pmatrix}
1 & 0\\0 & 1\end{pmatrix}
\begin{pmatrix}
x_1\\x_2\end{pmatrix}.$$

Algebra $\PP$:
$\Sigma=h\begin{pmatrix}
0&   0&  1&  f\\
0&   0&  1&  1\\
1&   -f&  0&  0\\
-1&   1&  0&  0
\end{pmatrix}$ and $M=h\begin{pmatrix}
0&   1&  0&  f\\
1&   0&  -f&  0\\
0&   1&  0&  1\\
-1&   0&  1&  0
\end{pmatrix}$ where $f\neq 1$; $P=Q=(-1,0)$. A
special case is 
when $f=0$. 

The algebra $\PP$ is $\Sigma$-M-dual to the algebra $\NN$.
The determinant of $\sigma$ is 
$$\det \sigma
\begin{pmatrix}
x_1\\x_2\end{pmatrix}
=(1-f)h^2\begin{pmatrix}
1 & 0\\0 & 1\end{pmatrix}
\begin{pmatrix}
x_1\\x_2\end{pmatrix}.$$

Algebra  $\QQ$:
$\Sigma=h\begin{pmatrix}
0&   0&  1&  0\\
1&   1&  1&  0\\
-1&   0&  0&  0\\
1&   0&  -1&  1
\end{pmatrix}$ and $M=h\begin{pmatrix}
0&   1&  0&  0\\
-1&   0&  0&  0\\
1&   1&  1&  0\\
1&   -1&  0&  1
\end{pmatrix}$; $P=Q=(-1,0)$. 

The determinant of $\sigma$ is 
$$\det \sigma
\begin{pmatrix}
x_1\\x_2\end{pmatrix}
=h^2\begin{pmatrix}
-1 & 0\\0 & 1\end{pmatrix}
\begin{pmatrix}
x_1\\x_2\end{pmatrix}.$$

Algebra $\RR$:
$\Sigma=M=h\begin{pmatrix}
1&   1&  1&  0\\
0&   0&  1&  0\\
0&   1&  0&  0\\
0&   -1&  -1&  1
\end{pmatrix}$; $P=Q=(-1,0)$. So the algebra $\RR$ is
$\Sigma$-M-selfdual. The determinant of $\sigma$ is 
$$\det \sigma
\begin{pmatrix}
x_1\\x_2\end{pmatrix}
=h^2\begin{pmatrix}
0 & 1\\-1 & 0\end{pmatrix}
\begin{pmatrix}
x_1\\x_2\end{pmatrix}.$$

Algebra $\SSS$:
$\Sigma=M=h\begin{pmatrix}
-1&   1&  1&  1\\
1&   -1&  1&  1\\
1&   1&  -1&  1\\
1&   1&  1&  -1
\end{pmatrix}$; $P=Q=(-1,0)$. The algebra $\SSS$ is
$\Sigma$-M-selfdual.
The determinant of $\sigma$ is 
$$\det \sigma
\begin{pmatrix}
x_1\\x_2\end{pmatrix}
=4h^2\begin{pmatrix}
1 & 0\\0 & 1\end{pmatrix}
\begin{pmatrix}
x_1\\x_2\end{pmatrix}.$$

Algebra $\TT$: 
$\Sigma=h\begin{pmatrix}
-1&   1&  1&  1\\
1&   -1&  1&  1\\
1&   1&  1&  -1\\
1&   1&  -1&  1
\end{pmatrix}$ and $M=h\begin{pmatrix}
-1&   1&  1&  1\\
1&   1&  1&  -1\\
1&   1&  -1&  1\\
1&   -1&  1&  1
\end{pmatrix}$; $P=Q=(-1,0)$.

The determinant of $\sigma$ is 
$$\det \sigma
\begin{pmatrix}
x_1\\x_2\end{pmatrix}
=4h^2\begin{pmatrix}
0 & 1\\1 & 0\end{pmatrix}
\begin{pmatrix}
x_1\\x_2\end{pmatrix}.$$

Algebra $\UU$: 
$\Sigma=h\begin{pmatrix}
-1&   1&  1&  1\\
1&   1&  1&  -1\\
1&   1&  -1&  1\\
1&   -1&  1&  1
\end{pmatrix}$ and $M=h\begin{pmatrix}
-1&   1&  1&  1\\
1&   -1&  1&  1\\
1&   1&  1&  -1\\
1&   1&  -1&  1
\end{pmatrix}$; $P=Q=(-1,0)$.
The algebras $\TT$ and $\UU$ are $\Sigma$-M-dual.
The determinant of $\sigma$ is 
$$\det \sigma
\begin{pmatrix}
x_1\\x_2\end{pmatrix}
=4h^2\begin{pmatrix}
1 & 0\\0 & 1\end{pmatrix}
\begin{pmatrix}
x_1\\x_2\end{pmatrix}.$$

\bigskip
\noindent
{\bf Subcase 4.3.4:} $Q=(1,0)$.
We can make linear transformations of $X$ to make that
$\Sigma_{12}$ is one of the standard forms:
$\Sigma_{12}=\begin{pmatrix}c_1&0\\0&c_1\end{pmatrix}$
or 
$\Sigma_{12}=\begin{pmatrix}c&0\\1&c\end{pmatrix}$. In
particular we may assume either $a_{1212}=0=a_{1221}$ 
or $a_{1212}=0$ and $a_{1221}=1$.

Let's consider the first case by assuming $a_{1212}=0
=a_{1221}$. If $a_{1211}=0=a_{1222}$, then 
$\Sigma_{12}=0$, we don't need to consider this.
Otherwise after exchanging $x_1$ and $x_2$, we may 
always assume that $a_{1211}\neq 0$. Replacing $y_i$ 
by scalar multiples, we may assume that
$a_{1211}=1$. In addition to those equivalent to 
the algebras $\KK$ and $\LL$,
the System C has two more solutions:

Algebra $\VV$:
$\Sigma=h\begin{pmatrix}
0&   1&  1&  0\\
0&   1&  0&  0\\
-1&   1&  0&  0\\
0&   0&  0&  1
\end{pmatrix}$ and $M=h\begin{pmatrix}
0&   1&  1&  0\\
-1&   0&  1&  0\\
0&   0&  1&  0\\
0&   0&  0&  1
\end{pmatrix}$; $P=(-1,0)$ and $Q=(1,0)$.

The determinant of $\sigma$ is 
$$\det \sigma
\begin{pmatrix}
x_1\\x_2\end{pmatrix}
=h^2\begin{pmatrix}
-1 & 1\\0 & 1\end{pmatrix}
\begin{pmatrix}
x_1\\x_2\end{pmatrix}.$$

Algebra $\WW$:
$\Sigma=h\begin{pmatrix}
0&   f&  1&  0\\
1&   0&  0&  -1\\
1&   0&  0&  f\\
0&   -1&  1&  0
\end{pmatrix}$ and $M=h\begin{pmatrix}
0&   1&  f&  0\\
1&   0&  0&  f\\
1&   0&  0&  -1\\
0&   1&  -1&  0
\end{pmatrix}$
where $f\neq -1$; $P=(-1,0)$ and $Q=(1,0)$.

The determinant of $\sigma$ is 
$$\det \sigma
\begin{pmatrix}
x_1\\x_2\end{pmatrix}
=(f+1)h^2\begin{pmatrix}
1 & 0\\0 & 1\end{pmatrix}
\begin{pmatrix}
x_1\\x_2\end{pmatrix}.$$

In the rest of Subcase 4.3.4, we assume that 
$a_{1212}=0$ and $a_{1221}=1$ and $a_{1211}=a_{1222}$.

The System C has two solutions:

Algebra $\XX$:
$\Sigma=h\begin{pmatrix}
0&   0&  1&  0\\
0&   0&  1&  1\\
1&   0&  0&  0\\
1&  1&  0&  0
\end{pmatrix}$ and $M=h\begin{pmatrix}
0&   1&  0&  0\\
1&   0&  0&  0\\
1&   1&  0&  1\\
1&  0&  1&  0
\end{pmatrix}$; $P=(-1,0)$ and $Q=(1,0)$.

The determinant of $\sigma$ is 
$$\det \sigma
\begin{pmatrix}
x_1\\x_2\end{pmatrix}
=h^2\begin{pmatrix}
1 & 0\\2 & 1\end{pmatrix}
\begin{pmatrix}
x_1\\x_2\end{pmatrix}.$$

Algebra $\YY$:
$\Sigma=h\begin{pmatrix}
1&   0&  0&  0\\
f&   -1&  1&  0\\
0&   0&  1&  0\\
1&   0&  f&  -1
\end{pmatrix}$ and $M=h\begin{pmatrix}
1&   0&  0&  0\\
0&   1&  0&  0\\
f&   1&  -1&  0\\
1&   f&  0&  -1
\end{pmatrix}$; $P=(-1,0)$ and $Q=(1,0)$.

The determinant of $\sigma$ is 
$$\det \sigma
\begin{pmatrix}
x_1\\x_2\end{pmatrix}
=h^2\begin{pmatrix}
1 & 0\\0 & 1\end{pmatrix}
\begin{pmatrix}
x_1\\x_2\end{pmatrix}.$$

\subsection{Case four: $P=(1,0)$.}  This is the
last piece of the classification. As before
we consider the following four subcases.

\bigskip
\noindent
{\bf Subcase 4.4.1:} $Q=(1,1)$. Up to linear
transformation all 
C-solutions give rise to iterated Ore extensions.

\bigskip
\noindent
{\bf Subcase 4.4.2:} $Q=(q,0)$ where $q\neq \pm 1$. Up
to linear transformation all 
C-solutions give rise to iterated Ore extensions.

\bigskip
\noindent
{\bf Subcase 4.4.3:} $Q=(1,0)$. All C-solutions
give rise to iterated Ore extensions up to 
linear transformation.

\bigskip
\noindent
{\bf Subcase 4.4.4:} $Q=(-1,0)$.
Up to linear transformation we have two
C-solutions which could lead to non-trivial
double extensions. The first one is
$\Sigma=h\begin{pmatrix}
-1&   0&  0&  0\\
0&   -1&  1&  0\\
0&   0&  1&  0\\
1&   0&  0&  1
\end{pmatrix}$ and $M=h\begin{pmatrix}
-1&   0&  0&  0\\
0&   1&  0&  0\\
0&   1&  -1&  0\\
1&   0&  0&  1
\end{pmatrix}$; $P=(1,0)$ and $Q=(-1,0)$. This is a special
case of the algebra $\DD$. The final case is

Algebra $\ZZ$:
$\Sigma=h\begin{pmatrix}
1&   0&  0&  1\\
0&   1&  1&  0\\
0&   f&  -1&  0\\
f&   0&  0&  -1
\end{pmatrix}$ and $M=h\begin{pmatrix}
1&   0&  0&  1\\
0&   -1&  f&  0\\
0&   1&  1&  0\\
f&   0&  0&  -1
\end{pmatrix}$
where $f(1+f)\neq 0$; $P=(1,0)$ and $Q=(-1,0)$.

The determinant of $\sigma$ is 
$$\det \sigma
\begin{pmatrix}
x_1\\x_2\end{pmatrix}
=-(f+1)h^2\begin{pmatrix}
1 & 0\\0 & 1\end{pmatrix}
\begin{pmatrix}
x_1\\x_2\end{pmatrix}.$$

When $f=-1$, the matrix $\Sigma$ is singular.
When $f=0$, then $\Sigma_{21}=0$. 
Note that the algebra $\ZZ$ is $\Sigma$-M-dual to 
the algebra $\WW$. To see this we need to use 
linear transformations. The $M$-matrix of the 
algebra $\ZZ$ can be changed to
$$M=h\begin{pmatrix}
0&   1&  \sqrt{f}&  0\\
1&   0&  0&  -\sqrt{f}\\
\sqrt{f}&   0&  0&  1\\
0&   -\sqrt{f}&  1&  0
\end{pmatrix}$$ 
after a linear transformation
$\begin{pmatrix} y_1\\y_2\end{pmatrix}
\to \begin{pmatrix} \sqrt{f}&1\\\sqrt{f}&-1
\end{pmatrix}\begin{pmatrix} y_1\\y_2\end{pmatrix}
$ which is equivalent to the $\Sigma$-matrix of
the algebra $\WW$ up to a linear transformation. 
This also shows that it is not obvious when two 
algebras are $\Sigma$-$M$-dual in general.

Finally we list all relations of the
algebra $\ZZ$:
$$\begin{aligned}
x_2x_1&=-x_1x_2\\
y_2y_1&=y_1y_2\\
y_1x_1&=x_1y_2+x_2y_2\\
y_1x_2&=x_2y_1+x_1y_2\\
y_2x_1&=fx_2y_1-x_1y_2\\
y_2x_2&=fx_1y_2-x_2y_2
\end{aligned}
$$

We summarize what we did in this section in
the following proposition that is also part (b) of
Theorem \ref{xxthm0.1}. We use 
$\LIST$ to denote the class consisting 
of all these 26 algebras from $\AAA$ to $\ZZ$.

\begin{proposition}
\label{xxprop4.4}
Suppose that $B:=
(k_{Q}[x_1,x_2])_{P}[y_1,y_2;\sigma,\delta,\tau]$ 
is a connected graded double extension with 
$x_1,x_2,y_1,y_2$ in degree $1$. If $B$ is not 
an iterated Ore extension of $k_{Q}[x_1,x_2]$, then 
the trimmed double extension
$(k_{Q}[x_1,x_2])_{P}[y_1,y_2;\sigma]$ is isomorphic to 
one of in the $\LIST$.
\end{proposition}

\begin{remark}
\label{xxrem4.5}
The algebras $\EE$ and $\JJ$ are $\Sigma$-$M$-dual. Consequently,
$\EE$ and $\JJ$ are isomorphic to each other by sending
$x_i$'s to $y_i$'s. However non-trimmed double
extensions extended from the algebras $\EE$ and $\JJ$ may not be 
isomorphic, because the roles played by
$x_i$'s and $y_i$'s are different in the non-trimmed
double extensions $(k_Q[x_1,x_2])_P[y_1,y_2;\sigma,\delta,\tau]$.
For the purpose of finding all non-trimmed double extensions
(which is another interesting project), we want to 
distinguish $\EE$ from $\JJ$ in our list. This remark applies to all 
$\Sigma$-$M$-dual pairs.
\end{remark}

\section{Properties of double extensions}
\label{xxsec5}

In this section we prove that all trimmed double 
extensions in the $\LIST$ classified in the
last section are strongly noetherian, Auslander regular
and Cohen-Macaulay. It seems to us that there is no
uniform method that works for all algebras, so we have to show 
this case by case. First we recall a result of \cite{ZZ}.

\begin{theorem}
\label{xxthm5.1}
\cite[Theorem 0.2]{ZZ}
Let $A$ be a regular algebra. Then any double extension 
$A_P[y_1,y_2;\sigma,\delta,\tau]$ is regular.  As a consequence, 
a double extension of the form  
$(k_{Q}[x_1,x_2])_P[y_1,y_2;\sigma,\delta,\tau]$
is regular.
\end{theorem}

This theorem ensures the Artin-Schelter regularity for 
algebras in the $\LIST$. For non-regular rings, it is 
convenient to use dualizing complexes which was introduced 
in \cite{Ye}. We will only use some facts about rings 
with dualizing complexes and refer to \cite{YZ} for 
definitions and properties and other details. 

An algebra $A$ is called {\it strongly 
noetherian} if for every commutative noetherian 
(ungraded) ring $S$, $A\otimes S$ is noetherian
\cite[p.580]{ASZ}. Let $A$ be a noetherian algebra
with a dualizing complex $R$. For any left 
$A$-module $M$, the grade of $M$ is defined to be
$$j_R(M)=\inf\{i\;|\; \Ext^i_A(M,R)\neq 0\}.$$
The grade of a right $A$-module is defined similarly.
The dualizing complex is called {\it Cohen-Macaulay} 
if there is a finite integer $d$ such that
$$j_R(M)+\GKdim M=d$$
for all finitely  generated left and right nonzero
$A$-modules $M$. The dualizing complex $R$ is called
{\it Auslander} \cite[Definition 2.1]{YZ} if the 
following conditions hold.
\begin{enumerate}
\item
for every finitely generated left $A$-module $M$, 
every integer $q$, and any right $A$-submodule
$N\subset \Ext^q_A(M,R)$, one has $j(N)\geq q$;
\item
the same holds after exchanging left with right.
\end{enumerate}
Since $A$ is connected graded, the balanced dualizing 
complex over $A$ (which is unique) is defined \cite{Ye}. 
For simplicity, we say $A$ has {\it Auslander property} 
(respectively, {\it Cohen-Macaulay property}) if (a) 
the balanced dualizing complex over $A$ exists and (b) 
the balanced dualizing complex of $A$ has the Auslander 
property (respectively, Cohen-Macaulay property).

If $A$ is regular (or Artin-Schelter 
Gorenstein), then the balanced dualizing complex
has the form of ${^\sigma A}(-l)[-n]$. In this case,
our definition of Auslander regular
and Cohen-Macaulay is equivalent to the 
usually definition given in \cite{Le}, namely,
$A$ is Auslander regular and Cohen-Macaulay (by taking $R=A$
in the definition) if and only if (a) $A$ is 
Artin-Schelter regular and (b) the balance
dualizing complex of $A$ has Auslander and Cohen-Macaulay 
properties. One of the usefulness of dualizing complexes is 
that Auslander and Cohen-Macaulay properties are
defined for non-regular rings. For example, the Auslander 
and Cohen-Macaulay properties pass from a graded ring
to any of its factor rings without worrying 
the regularity. 

As we have seen that many double extensions are iterated 
Ore extensions. Those algebras are Auslander (and 
Auslander regular) and Cohen-Macaulay by the following 
lemma.

\begin{lemma}
\label{xxlem5.2} Let $B:=A[t;\sigma,\delta]$ be a connected 
graded Ore extension of a noetherian algebra $A$.
\begin{enumerate}
\item
If $A$ is strongly noetherian, then so is $B$.
\item
If $A$ is Auslander and Cohen-Macaulay, so is $B$.
\item
If $A$ is regular of dimension three, then 
$A$ is strongly noetherian and Auslander regular 
and Cohen-Macaulay.
\item
If $A$ is commutative (or PI), then $A$ is strongly 
noetherian and Auslander regular and Cohen-Macaulay.
\end{enumerate}
\end{lemma}

\begin{proof} (a) This is \cite[Proposition 4.1(b)]{ASZ}.

(b) We construct a connected graded noetherian 
filtration on $B$ by setting new degree of $t$ to be 
$\deg t+1$, so the associated graded ring of $B$ has
the property $\gr B\cong A[t;\sigma]$. 
By \cite[Corollary 6.8]{YZ}, we only need to show the
assertion for $A[t;\sigma]$. Then it follows from 
\cite[Theorem 5.1]{YZ}.

(c) For the strongly noetherian property, we note that
there is a normal element $g$ of degree 3 such that 
$A/(g)$ is noetherian of $\GKdim$ $2$. By 
\cite[Theorem 4.24]{ASZ}, $A/(g)$ is strongly noetherian,
and by \cite[Proposition 4.9]{ASZ}, $A$ is strongly 
noetherian. The rest is \cite[Corollary 6.2]{Le}.

(d) See \cite[Proposition 4.9(5)]{ASZ} and 
\cite[Corollary 6.9(i)]{YZ}. 
\end{proof}

As seen in the last section, many algebras are Ore extensions
of regular algebras of dimension three. By
the above lemma, the following proposition
is proved.

\begin{proposition}
\label{xxprop5.3} 
Algebras $\AAA$, $\DD$, $\GG$, $\HH$, $\KK$, $\LL$, $\QQ$,
$\VV$, $\XX$ and $\YY$ are 
strongly noetherian, Auslander regular and 
Cohen-Macaulay.
\end{proposition} 

\begin{proof} In each of these cases, we have either
$M_{12}=0$ or $M_{21}=0$. So every algebra can be
written as an Ore extension of a regular algebra of 
dimension three. 
The assertion follows from  Lemma \ref{xxlem5.2}.
\end{proof}

The following is proved in \cite{ZZ}.

\begin{lemma}
\cite[Proposition 0.5 and Section 4]{ZZ}
\label{xxlem5.4} The algebra $\RR$ is strongly 
noetherian, Auslander regular and Cohen-Macaulay. 
\end{lemma}

It is clear that $B$ is strongly noetherian, Auslander 
regular and Cohen-Macaulay if and only if the opposite
ring $B^{op}$ is. If two double extensions are 
$\Sigma$-M-dual, then they are opposite to each other
up to equivalences (see Proposition \ref{xxprop3.7}
and Definition \ref{xxdefn4.3}). Therefore we have
the following.

\begin{lemma}
\label{xxlem5.5}
If algebras $A$ and $B$ are $\Sigma$-M-dual,
then $A$ is strongly noetherian, Auslander 
(or Auslander regular) and Cohen-Macaulay if and only 
if $B$ is.
\end{lemma}

So we can pair together $\Sigma$-M-dual algebras:
$(\EE,\JJ)$, $(\FF,\II)$, $(\NN,\PP)$, $(\TT,\UU)$
and $(\WW,\ZZ)$. Other algebras $\Sigma$-M-selfdual:
$\BB,\CC,\MM,\OO,\RR, \SSS$. (See Lemma \ref{xxlem5.4}
for the algebra $\RR$). Basically it reduces to ten 
algebras to work on.

The following three lemma are well-known.

\begin{lemma}
\label{xxlem5.6} 
Let $A$ be a connected graded 
algebra ane let $t$ be a homogeneous normal
element of $A$ (not necessarily a nonzerodivisor).
\begin{enumerate}
\item
\cite[Proposition 4.9]{ASZ}
$A$ is noetherian (respectively, strongly noetherian)
if and only if $A/(t)$ is.
\item
\cite[Theorem 5.1]{YZ}
Suppose $A$ is noetherian. Then $A$ is Auslander
(respectively, strongly noetherian, Cohen-Macaulay) 
if and only if $A/(t)$ is.
\end{enumerate}
\end{lemma}

\begin{lemma}
\label{xxlem5.7} 
\cite[Proposition 3.9]{YZ}
Let $A$ and $B$ be a connected graded algebras. 
\begin{enumerate}
\item
Suppose $A\subset B$ and $B_A$ and $_AB$ are finite. 
If $A$ is strongly noetherian, Auslander and 
Cohen-Macaulay, then so is $B$.
\item
Suppose $B$ is a factor ring $A/I$. If $A$ is 
strongly noetherian, Auslander and 
Cohen-Macaulay, then so is $B$.
\end{enumerate}
\end{lemma}

\begin{lemma}
\label{xxlem5.8} 
Let $B$ be a graded twist of $A$ in the
sense of \cite{Zh1}. Then $A$ is strongly noetherian, 
Auslander (or Auslander regular) and 
Cohen-Macaulay if and only if $B$ is.
\end{lemma}

The usefulness of Lemma \ref{xxlem5.8} is that if
two algebras $A$ and $B$ are twist-equivalent in the sense of 
Definition \ref{xxdefn3.3}, then $A$ is strongly noetherian, 
Auslander (or Auslander regular) and 
Cohen-Macaulay if and only if $B$ is. In particular,
when we prove a double extension is strongly noetherian, 
Auslander (or Auslander regular) and Cohen-Macaulay,
we may assume that $h=1$, which we will do in the rest
of this section.

Let $A$ be a graded ring and $n$ be a positive integer.
The $n$th {\it Veronese subring} of $A$ is defined to be
$$A^{(n)}=\oplus_{i\in {\mathbb Z}} A_{in}.$$

For later discussion we will use the following special
case of Lemma \ref{xxlem5.7}(a).

\begin{lemma}
\label{xxlem5.9} 
Let $k_Q[x_1,x_2]^{(2)}$ be the 2nd Veronese subring 
of $k_Q[x_1,x_2]$. Then the algebra $(k_Q[x_1,x_2])_P
[y_1,y_2;\sigma]$ has the noetherian property if and 
only if the subalgebra $(k_Q[x_1,x_2]^{(2)})_P
[y_1,y_2;\sigma]$ does. Same for strongly noetherian,
Auslander, Cohen-Macaulay properties.
\end{lemma}

The algebra $(k_Q[x_1,x_2]^{(2)})_P[y_1,y_2;\sigma]$
is a double extension of $k_Q[x_1,x_2]^{(2)}$ where
$\sigma$ is the restriction of $\sigma$ on the 
2nd Veronese subring. The next lemma is 
particularly useful for the algebras such as $\BB$.

\begin{lemma}
\label{xxlem5.10} 
If $\Sigma_{11}=\Sigma_{22}=0$, then 
the double extension is strongly noetherian, 
Auslander and Cohen-Macaulay.
\end{lemma}

\begin{proof} In this case $\sigma_{11}=0=\sigma_{22}$.
For every $a\in kx_1+kx_2$ we have
$$y_1a=\sigma_{12}(a)y_2\quad
\text{and}\quad y_2a=\sigma_{21}(a)y_1.$$
Hence for every $a,b\in kx_1+kx_2$, 
$$y_1ab=\sigma_{12}(a)\sigma_{21}(b)y_1
\quad{\text{and}}\quad
y_2ab=\sigma_{21}(a)\sigma_{12}(b)y_2.$$
Hence $y_1$ and $y_2$ are normal elements in 
$(k_Q[x_1,x_2])^{(2)}_P[y_1,y_2;\sigma]$.

It is well-known that $(k_Q[x_1,x_2])^{(2)}$
is strongly noetherian, Auslander and 
Cohen-Macaulay. By Lemma \ref{xxlem5.6}(a), so 
is $(k_Q[x_1,x_2])^{(2)}_P[y_1,y_2;\sigma]$. The 
assertion follows from Lemma \ref{xxlem5.9}.
\end{proof}

Here is a consequence of Lemma \ref{xxlem5.10}.

\begin{proposition}
\label{xxprop5.11} 
The algebras $\BB$, $\EE$
$\JJ$, $\NN$ and $\PP$ are strongly noetherian, 
Auslander and Cohen-Macaulay.
\end{proposition}

\begin{proof} 
Lemma  \ref{xxlem5.10} is applied directly to the algebra 
$\BB$, $\EE$ and $\PP$. For the algebras $\JJ$ and
$\NN$, use $\Sigma$-$M$-dual property (see Lemma \ref{xxlem5.5})
and then apply Lemma \ref{xxlem5.10}.
\end{proof}

We will leave the algebra $\CC$ to the end and work on 
other algebras first.

\begin{proposition}
\label{xxprop5.12}
The algebras $\FF$ and $\II$ are strongly noetherian, 
Auslander and Cohen-Macaulay.
\end{proposition}

\begin{proof} Since the algebra $\FF$ is $\Sigma$-$M$-dual
to the algebra $\II$, we only consider $\II$. 

The relations of the algebra $\II$ are
$$
\begin{aligned}
x_2x_1&=q x_1x_2\\
y_2y_1&=-y_1y_2\\
y_1x_1&=-qx_1y_1-qx_2y_1+x_1y_2-qx_2y_2\\
y_1x_2&=x_1y_1+x_2y_1+x_1y_2-qx_2y_2\\
y_2x_1&=x_1y_1+qx_2y_1+qx_1y_2-qx_2y_2\\
y_2x_2&=-x_1y_1-qx_2y_1+x_1y_2-x_2y_2
\end{aligned}$$
where $q^2=-1$. We have assumed $h=1$ by using
Lemma \ref{xxlem5.8}. 
Using these relations we obtain the 
following relations
$$
\begin{aligned}
y_1(x_1-x_2)&=(-1-q)(x_1+x_2)y_1\\
y_1(x_1+qx_2)&=(1+q)(x_1-qx_2)y_2\\
y_2(x_1+x_2)&=(1+q)(x_1-x_2)y_2\\
y_2(x_1-qx_2)&=(1+q)(x_1+qx_2)y_2
\end{aligned}
$$
Using the first and the third relations above
we have
$$
\begin{aligned}
y_1y_2(x_1-x_2)&=-y_2y_1(x_1-x_2)=(1+q)y_2(x_1+x_2)y_1\\
               &=(1+q)^2(x_1-x_2)y_2y_1
=-(1+q)^2(x_1-x_2)y_1y_2
\end{aligned}
$$
A similar computation shows that
$$y_1y_2(x_1+x_2)=-(1+q)^2(x_1+x_2)y_1y_2$$
Hence $y_1y_2$ is skew-commuting with $x_i$'s
with the scalar $-(1+q)^2$. So it is a normal
element in $D:=(k_P[y_1,y_2]^{(2)})_Q[x_1,x_2]$.

Using the original relations one also sees that
$$
\begin{aligned}
(y_1+qy_2)x_1&=(-1-q)x_2(y_1+qy_2)\\
(y_1-qy_2)x_2&=(1+q)x_1(y_1-qy_2)
\end{aligned}
$$
Then we have
$$(y_1+qy_2)(y_1-qy_2)x_2=(1+q)(y_1+qy_2)x_1
(y_1-qy_2)=-(1+q)^2x_2(y_1+qy_2)(y_1+qy_2)$$
or
$$(y_1^2+y^2-2qy_1y_2)x_2=-(1+q)^2
x_2(y_1^2+y_2^2-2qy_1y_2)$$
Using the relation 
$$y_1y_2x_2=-(1+q)^2 x_2y_1y_2$$
which was proved in the last paragraph,
we obtain
$$(y_1^2+y_2^2)x_2=-(1+q)^2 x_2(y_1^2+
y_2^2).$$
By symmetry,
$$(y_1^2+y_2^2)x_1=-(1+q)^2 x_1(y_1^2+
y_2^2).$$
Hence both $y_1y_2$ and $y_1^2+y_2^2$
are normal elements in $D$. The factor
ring $D/(y_1y_2,y_1^2+y_2^2)$ is a finite module 
over $k_Q[x_1,x_2]$. The assertion follows from 
Lemmas \ref{xxlem5.6}, \ref{xxlem5.7} and \ref{xxlem5.9}.
\end{proof}

\begin{proposition}
\label{xxprop5.13}
The algebras $\MM$ and $\OO$ are strongly noetherian, 
Auslander and Cohen-Macaulay.
\end{proposition}

\begin{proof} We first consider the algebra $\OO$, then 
sketch for the algebra $\MM$. 

The mixing relations of the algebra $\OO$ are
$$
\begin{aligned}
y_1x_1&=x_1y_1+fx_2y_2\\
y_1x_2&=-x_2y_1+x_1y_2\\
y_2x_1&=fx_2y_1-x_1y_2\\
y_2x_2&=x_1y_1+x_2y_2.
\end{aligned}
$$
Using these we obtain

\begin{align}
\label{E5.11.1}
y_1x_1^2&=(x_1y_1+fx_2y_2)x_1\tag{E5.11.1}\\
&=x_1(x_1y_1+fx_2y_2)+fx_2(fx_2y_1-x_1y_2)\notag\\
&=(x_1^2+f^2x_2^2)y_1+2fx_1x_2y_2.\notag
\end{align}
Similarly we have the following
\begin{align}
\label{E5.11.2}
y_1x_2^2&=(x_1^2+x_2^2)y_1+2x_1x_2y_2,
\tag{E5.11.2}\\
\label{E5.11.3}
y_1x_1x_2&=(-1-f)x_1x_2y_1+(x_1^2+fx_2^2)y_2,
\tag{E5.11.3}\\
\label{E5.11.4}
y_2x_1^2&=-2fx_1x_2y_1+(x_1^2+f^2x_2^2)y_2,
\tag{E5.11.4}\\
\label{E5.11.5}
y_2x_2^2&=-2x_1x_2y_1+(x_1^2+x_2^2)y_2,
\tag{E5.11.5}\\
\label{E5.11.6}
y_2x_1x_2&=(-x_1^2-fx_2^2)y_1+(-1-f)x_1x_2y_2.
\tag{E5.11.6}
\end{align}
A linear combination of \eqref{E5.11.1} and \eqref{E5.11.2}
gives rise to
$$y_1(x_1-fx_2^2)=(1-f)(x_1^2-fx_2^2)y_1$$
and a linear combination of \eqref{E5.11.4} and \eqref{E5.11.5}
gives rise to
$$y_2(x_1-fx_2^2)=(1-f)(x_1-fx_2^2)y_2.$$ 
Therefore $x_1^2-fx_2^2$ is a normal element in the algebra
$\OO$. By Lemma \ref{xxlem5.9} we only need to show
that $D':=k_Q[x_1,x_2]^{(2)}_P[y_1,y_2]$ has  the properties
stated in Proposition \ref{xxprop5.13}. By Lemma 
\ref{xxlem5.6}(b) we only need to show that the factor 
ring $D:=D'/(x_1^2-fx_2^2)$ has the desired properties.
Now let us introduce some new variables: $X_1=x_1^2$ and 
$X_2=x_1x_2$. Then we have $X_1X_2=X_2X_1$, $y_1y_2=-y_2y_1$; 
and, after identifying $fx_2^2$ with $X_1$ in the
algebra $D$, \eqref{E5.11.2}-\eqref{E5.11.6} imply 
$$\begin{aligned}
y_1 X_1&=(1+f)X_1 y_1+2f X_2 y_2,\\
y_1 X_2&=(-1-f) X_2 y_1+2 X_1 y_2,\\
y_2 X_1&=-2f X_2 y_1   +(1+f) X_1 y_2,\\
y_2 X_2&=-2 X_1 y_1+(-1-f) X_2 y_2.
\end{aligned}
$$
Use these relations we see that 
$$(y_1+iy_2)X_1=((1+f)X_1-2ifX_2)(y_1+iy_2),$$
and
$$(y_1+iy_2)X_2=(-2iX_1-(1+f)X_2)(y_1+iy_2).$$
So $y_1+iy_2$ is skew-commuting with $X_i$'s.
Similarly, $y_1-iy_2$ is skew-commuting
with $X_i$s. Thus $(y_1-iy_2)^2$,
$(y_1+iy_2)(y_1-iy_2)$ and $(y_1-iy_2)(y_1+iy_2)$
are all skew-commuting with $X_i$s.
Note that $\{(y_1-iy_2)^2,(y_1+iy_2)(y_1-iy_2),
(y_1-iy_2)(y_1+iy_2)\}$ is a $k$-linear basis
for $k_P[y_1,y_2]_2$. Since $y_1y_2=-y_2y_1$, 
the subalgebra of $C\subset D$ generated by 
$X_1, X_2$ and the degree two elements in 
$k_P[y_1,y_2]$ are in fact generated by five 
normal elements in $C$. By Lemma \ref{xxlem5.6} 
$C$ has the desired properties. 
Since $B$  is a finite module over $C$,
$B$ has the desired properties. This finishes
the case $\OO$. 

The proof for Case $\MM$ is similar. The details are
omitted since we can use the proof in the case
of the algebra $\OO$. Let us only give a few key
points. We work 
inside the ring $D':=k_Q[x_1,x_2]^{(2)}_P[y_1,y_2]$.
First we show that $fx_1^2-x_2^2$ is normal. 
Modulo $fx_1^2-x_2$ in $D'$, elements $y_1+iy_2$
and $y_1-iy_2$ are skew-commuting with
$X_1:=x_2^2$ and $X_2=x_1x_2$. Therefore the algebra $\MM$ is
strongly noetherian, Auslander and Cohen-Macaulay. 
\end{proof}

Near the end of the proof of Proposition \ref{xxprop5.13},
we use the fact $y_1\pm iy_2$ are skew-commuting with
$X_i$. This ideas can be used for the algebra $\SSS$.
We say an element $x$ is skew commutative
with $y_1$ and $y_2$ in side a ring $D$ if $x(ky_1+ky_2)=
(ky_1+ky_2)x$ holds in $D$. 

\begin{lemma}
\label{xxlem5.14}
Let $B$ be an algebra $(k_{Q}[x_1,x_2])_{P}[y_1,y_2;\sigma]$
with $Q=(-1,0)$. If $x_1-x_2$ and $x_1+x_2$ are
skew commutating with $y_i$'s, then $B$ is strongly 
noetherian, Auslander and Cohen-Macaulay.
\end{lemma}

\begin{proof} Let 
$$\begin{aligned}
a_1&=(x_1+x_2)^2=x_1^2+x_2^2, \\
a_2&=(x_1-x_2)(x_1+x_2)=x_1^2+2x_1x_2-x_2^2\\
a_3&=(x_1+x_2)(x_1-x_2)=x_1^2-2x_1x_2-x_2^2.
\end{aligned}
$$
Hence $\{a_1,a_2,a_3\}$ is a $k$-linear basis
of $k_Q[x_1,x_2]_2$. By hypotheses, $x_1-x_2$ and 
$x_1+x_2$ are skew-commuting with $y_1,y_2$. Hence 
$a_1,a_2,a_3$ are normalizing elements in $D:=
(k_Q[x_1,x_2]^{(2)})_P[y_1,y_2;\sigma]$.
Clearly, the factor ring $D/(a_1,a_2,a_3)$ is 
isomorphic to $k_P[y_1,y_2]$, which is
strongly noetherian, Auslander and
Cohen-Macaulay. The assertion follows
from Lemmas \ref{xxlem5.6},  \ref{xxlem5.7} and \ref{xxlem5.9}.
\end{proof}

\begin{proposition}
\label{xxprop5.15} The algebras $\SSS$, $\TT$, $\UU$, $\WW$ and
$\ZZ$ is strongly noetherian, Auslander and
Cohen-Macaulay.
\end{proposition}

\begin{proof} First we consider the algebra $\TT$. The 
four mixing relations of the algebra $\TT$ are
$$\begin{aligned}
y_1x_1&=-x_1y_1+x_2y_1+x_1y_2+x_2y_2\\
y_1x_2&=x_1y_1-x_2y_1+x_1y_2+x_2y_2\\
y_2x_1&=x_1y_1+x_2y_1+x_1y_2-x_2y_2\\
y_2x_2&=x_1y_1+x_2y_1-x_1y_2+x_2y_2.
\end{aligned}$$
Using these we obtain that
$$y_1(x_1+x_2)=2(x_1+x_2)y_2
\quad{\text{ and}}\quad
y_2(x_1+x_2)=2(x_1+x_2)y_1.$$
Hence $x_1+x_2$ is skew-commuting with $y_i$s.
Similarly, $x_1-x_2$ is skew-commuting $y_i$s. 
The assertion for the algebra $\TT$ 
follows from Lemma \ref{xxlem5.14}.

Since the algebra $\UU$ is $\Sigma$-$M$-dual to
the algebra $\TT$. The assertion for the algebra 
$\UU$ follows from Lemma \ref{xxlem5.5}.

The proof for the algebra $\SSS$ is very similar to 
the proof for the algebra $\TT$. So the details are
omitted.

For the algebra $\ZZ$, we use Lemma \ref{xxlem5.14}
again. Note that $Q=(-1,0)$. The mixing relations 
of this algebra imply that
$$
\begin{aligned}
y_1(x_1+x_2)&=(x_1+x_2)(y_1+y_2)\\
y_2(x_1+x_2)&=(x_1+x_2)(fy_1-y_2)\\
y_1(x_1-x_2)&=(x_1-x_2)(y_1-y_2)\\
y_2(x_1-x_2)&=(x_1-x_2)(-fy_1-y_2).
\end{aligned}
$$
These relations show that $x_1-x_2$ and
$x_1+x_2$ are skew-commuting with $y_i$.
The assertion follows from Lemma \ref{xxlem5.14}.

Since the algebra $\WW$ is $\Sigma$-$M$-dual to
the algebra $\ZZ$. The assertion for $\WW$ follows 
from Lemma \ref{xxlem5.5}.
\end{proof}

The last case to deal with is the algebra $\CC$, 
which is slightly more complicated.

\begin{proposition}
\label{xxprop5.16} 
The algebra $\CC$ is strongly noetherian, Auslander and
Cohen-Macaulay.
\end{proposition}

\begin{proof}
First we list the relations of the algebra $\CC$ as
follows:
$$y_2y_1=py_1y_2\quad\text{and}\quad
x_2x_1=px_1x_2$$
where $p^2+p+1=0$ (or $p^3=1$ and $p\neq 1$); and
$$
\begin{aligned}
y_1x_1&=-x_1y_1+p^2 x_2y_1+x_1y_2-px_2y_2\\
y_1x_2&=-px_1y_1+  x_2y_1+x_1y_2-px_2y_2\\
y_2x_1&=-px_1y_1-2p^2 x_2y_1+px_1y_2-px_2y_2\\
y_2x_2&=-px_1y_1+p^2 x_2y_1+x_1y_2-x_2y_2.
\end{aligned}
$$
Using these relations we obtain the following
$$
\begin{aligned}
y_1(x_1-x_2)&=(p-1)(x_1-p^2x_2)y_1\\
y_1(x_1-p^2x_2)&=(1-p^2)(x_1-px_2)y_2\\
y_2(x_1-px_2)&=(p^2-p)(x_1-x_2)y_1.
\end{aligned}
$$
These three relations are used in the following computations:
$$
\begin{aligned}
y_1^2y_2 (x_1-px_2)&=(p^2-p)y_1^2(x_1-x_2)y_1\\
&=(p^2-p)(p-1)y_1(x_1-p^2x_2)y_1^2\\
&=(p^2-p)(p-1)(1-p^2)(x_1-px_2)y_2y_1^2\\
&=(p-1)^2(1-p^2)(x_1-px_2)y_1^2y_2.
\end{aligned}
$$
The same three relations also imply 
$$y_2y_1^2(x_1-x_2)=(p-1)^2(1-p^2)(x_1-x_2)
y_2y_1^2.$$
Since $y_2y_1^2=p^2y_1^2y_2$, $y_1^2y_2$ is 
skew-commuting with $x_1$ and $x_2$ with scalar 
$(p-1)^2(1-p^2)$. Since $y_1^2y_2$ is skew-commuting
with $y_1$ and $y_2$, it is a normalizing element in
the algebra $\CC$. 

Next we will find another normalizing element in
degree 3.  Using the four mixing relations we obtain three 
other relations:
$$
\begin{aligned}
(y_1-y_2)x_2&=(1-p^2)x_2(y_1-py_2)\\
(y_1-py_2)x_2&=(p^2-p)x_1(y_1-p^2y_2)\\
(y_1-p^2y_2)x_1&=(p-1)x_2(y_1-y_2).
\end{aligned}
$$
The first relations of these three shows that
$$(y_1-y_2)^3x_2=(1-p^2)^3x_2(y_1-py_2)^3.$$
It is easy to show that
$$(y_1-y_2)^3=(y_1-py_2)^3=(y_1-p^2y_2)^3=y_1^3-y_2^3.$$
Thus $y_1^3-y_2^3$ is skew-commuting with $x_2$
with scalar $(1-p^2)^3$. Using all three relations
we obtain that
$$\begin{aligned}
(y_1-py_2) & (y_1-y_2)(y_1-p^2y_2)x_1 \qquad\qquad\qquad\qquad \\
&=(p-1)(y_1-py_2)(y_1-y_2)x_2(y_1-y_2)\\
&=(p-1)(1-p^2)(y_1-py_2)x_2(y_1-py_2)(y_1-y_2)\\
&=(p-1)(1-p^2)(p^2-p)x_1(y_1-p^2y_2)(y_1-py_2)(y_1-y_2).
\end{aligned}
$$
An easy computation shows that
$$(y_1-py_2)(y_1-y_2)(y_1-p^2y_2)=(y_1-p^2y_2)(y_1-py_2)(y_1-y_2)=
y_1^3-y_2^3.$$
Therefore $y_1^3-y_2^3$ is skew-commuting with
$x_1$. Since $y_1^3-y_2^3$ is commuting with $y_i$'s.
We conclude that $y_1^3-y_2^3$ is a normalizing  
element in the algebra $\CC$. After factoring out 
both elements $y_1^2y_2$ and $y_1^3-y_2^3$ in $\CC$,
the factor ring is finite module over $k_Q[x_1,x_2]$.
By Lemmas \ref{xxlem5.6} ,\ref{xxlem5.7} and \ref{xxlem5.9}, the
algebra $\CC$ is strongly noetherian, Auslander and
Cohen-Macaulay.
\end{proof}

Combining these propositions we have

\begin{theorem}
\label{xxthm5.17}
The algebras $\AAA$ to $\ZZ$ are strongly noetherian, Auslander and
Cohen-Macaulay.
\end{theorem}

We are (almost) ready to prove Theorem \ref{xxthm0.1}. We refer
to \cite[Section 6]{YZ} for the definitions related to 
noetherian filtrations.

\begin{lemma}
\label{xxlem5.18}
Let $A$ be a filtered algebra such that the associated graded
ring $\gr A$ is strongly noetherian, Auslander and
Cohen-Macaulay. Then so is $A$.
\end{lemma}

\begin{proof} By \cite[Proposition 4.10]{ASZ}, $A$ is strong 
noetherian. The rest follows from \cite[Corollary 6.8]{YZ}.
\end{proof}

\begin{proof}[Proof of Theorem \ref{xxthm0.1}]
The regularity follows from Theorem \ref{xxthm5.1}.

(c) This is Proposition \ref{xxprop4.4}.

(a) Let $B$ be a double  extension $A_{P}[y_1,y_2;\sigma,\delta,
\tau]$ where $A=k_{Q}[x_1,x_2]$. If $B$ is an iterated Ore 
extension of $A$, then the assertion follows from 
Lemma \ref{xxlem5.2}. Now we assume $B$ is not an iterated
Ore extension of $A$. By \cite[Lemma 4.4]{ZZ}, $B$ has a
filtration such that $\gr B$ is the trimmed double
extension. By Lemma \ref{xxlem5.18}, it suffices to show 
that the trimmed Ore extension is strongly noetherian, 
Auslander and Cohen-Macaulay. By part (c),
since $B$ is not an iterated Ore extension of $A$, 
the trimmed double extension is one of the algebras $\AAA$ 
to $\ZZ$. Therefore the assertion follows from Theorem 
\ref{xxthm5.17}.

(b) This is \cite[Proposition 1.4]{LPWZ}.
\end{proof}

Recall that a regular algebra $B$ is called a normal extension if 
there is a non-zero-divisor $x$ of degree 1 such that $B/(x)$ 
is Artin-Schelter regular. Many of the algebras in the $\LIST$ 
are not isomorphic to either Ore extensions or normal extension 
of regular algebras of dimension three. This can be 
proved by using the method in the proof of 
\cite[Lemmas 4.9 and 4.10]{ZZ}. 

\section*{Acknowledgments}

James J. Zhang is supported by the National Science 
Foundation of USA and the Royalty Research Fund of the
University of Washington.

\providecommand{\bysame}{\leavevmode\hbox
to3em{\hrulefill}\thinspace}
\providecommand{\MR}{\relax\ifhmode\unskip\space\fi MR
}
\providecommand{\MRhref}[2]{%
 
\href{http://www.ams.org/mathscinet-getitem?mr=#1}{#2}
}
\providecommand{\href}[2]{#2}

\end{document}